\documentclass[12pt]{article}
\usepackage{amsmath,amssymb,amsfonts,amsthm,graphicx,color}
\usepackage{mathrsfs}
\usepackage{hyperref}
\usepackage{mathtools}
\usepackage{upgreek}
\usepackage{graphicx,type1cm,eso-pic,color}
\usepackage[margin=0.9in]{geometry}
\allowdisplaybreaks
\title{Global existence and non-existence of weak solutions for non-local stochastic semilinear reaction-diffusion equations driven by a fractional noise}
\author{S. Sankar${}^{1}$, Manil T. Mohan${}^{2,}\thanks{Corresponding author email: maniltmohan@ma.iitr.ac.in, maniltmohan@gmail.com.}$  and S. Karthikeyan${}^{1}$\\
	\footnotesize{$^1$Department of Mathematics, Periyar University, Salem 636 011, India}\\
	\footnotesize{$^2$Department of Mathematics, Indian
	Institute of Technology Roorkee, Roorkee 247 667, India}}
\date{}
\allowdisplaybreaks

\usepackage{xcolor}
\allowdisplaybreaks

\usepackage[pagewise]{lineno}

\usepackage{graphicx,eurosym}
\usepackage{hyperref}
\usepackage{mathtools}

\colorlet{darkblue}{blue!50!black}

\hypersetup{
	colorlinks,%
	citecolor=blue,%
	filecolor=red,%
	linkcolor=darkblue,%
	urlcolor=blue,%
	pdfnewwindow=true,%
	pdfstartview={FitH}
}

\usepackage{graphicx,amscd,mathrsfs,wrapfig,mathrsfs,lipsum}
\usepackage{eufrak}
\usepackage{tikz}
\usepackage{multicol}
\usepackage{caption}
\usetikzlibrary{arrows}

\colorlet{darkblue}{red!100!black}

\let\originalleft\left
\let\originalright\right
\renewcommand{\left}{\mathopen{}\mathclose\bgroup\originalleft}
\renewcommand{\right}{\aftergroup\egroup\originalright}

\begin{document}
	\maketitle \setcounter{page}{1} \numberwithin{equation}{section}
	\newtheorem{theorem}{Theorem}[section]
	\newtheorem{assumption}{Assumption}
	\newtheorem{lemma}{Lemma}[section]
	\newtheorem{Pro}{Proposition}[section]
	\newtheorem{Ass}{Assumption}[section]
	\newtheorem{Def}{Definition}[section]
	\newtheorem{Rem}{Remark}[section]
	\newtheorem{corollary}{Corollary}[section]
	\newtheorem{proposition}{Proposition}[section] 
	\newtheorem{app}{Appendix:}
	\newtheorem{ack}{Acknowledgement:}

	\begin{abstract}
	In the present paper, we study the existence and blow-up behavior to the following stochastic non-local reaction-diffusion equation:  
	\begin{equation*}
	\left\{
	\begin{aligned}
	du(t,x)&=\left[(\Delta+\gamma) u(t,x)+\int_{D}u^{q}(t,y)dy -ku^{p}(t,x)+\delta u^{m}(t,x)\int_{D}u^{n}(t,y)dy \right]dt\\
	&\quad+\eta u(t,x)dB^{H}(t),\\
	u(t,x)&=0, \ \ t>0, \ \ x\in \partial D, \\
	u(0,x)&=f(x) \geq 0, \ \ x\in D,\\
	\end{aligned}
	\right.
	\end{equation*}
	where $D\subset \mathbb{R}^{d}\ (d \geq 1)$ is a bounded domain  with smooth boundary $\partial D$. Here, $k>0, \gamma, \delta, \eta \geq 0$ and $p,q,n>1,\ m\geq 0$ with $m+n \geq q\geq p$. The initial data $f $ is a non-negative bounded measurable function in class $C^{2}$ which is not identically zero. Here, $\left\{ B^{H}(t) \right\}_{t \geq 0} $ is a one-dimensional fractional Brownian motion with Hurst parameter $\frac{1}{2} \leq H<1$ defined on a filtered probability space $\left( \Omega, \mathcal{F}, (\mathcal{F}_{t})_{t \geq 0}, \mathbb{P} \right)$. First, we estimate a lower bound for the finite-time blow-up and by choosing a suitable initial data, we obtain the upper bound for the finite-time blow-up of the above equation. Next, we provide a sufficient condition for the global existence of a weak solution of the above equation. Further, we obtain the bounds for the probability of blow-up solution.\\

		
	\end{abstract}
	\section{Introduction}
	 A number of physical phenomena, including chemical reactions, electrical heating, and fluid flow, give rise to nonlinear diffusion problems that are modeled as non-local reaction-diffusion equations \cite{pao}. In many applications in engineering and biology, blow-up behaviour is associated with some destructive behavior of the mathematical models. Thus from the mathematical point of view, it is very interesting to investigate the conditions under which such a finite-time blow-up occurs.  In the past few decades, many authors have discussed the global existence and blow-up of solutions for local and non-local reaction-diffusion equations, and there have been many monographs as well as review papers \cite{beber,hu,quit} etc.  Kaplan \cite{kaplan} and Fujita \cite{fuji1966, Fuji1968} and followed by them, many researchers have developed the concept of blow-up  to a more general class of non-linear partial differential equations (PDEs).  Over the past few years, a considerable effort has been devoted to study the blow-up properties of solutions to parabolic equations with local boundary conditions, say Dirichlet, Neumann or Robin boundary conditions, which can be used to describe heat propagation on the boundary, see the papers (cf.\cite{deny,levine} and references therein).
	 
	 Non-local mathematical processes appear in different physical forms (see \cite{beber192,day,day193}).  Bebernes et al. in \cite{beber19} obtained the blow-up results for a more general class of non-local reaction diffusion equations of the form in a bounded domain $D\subset  \mathbb{R}^d\ (d \geq 1)$ with suffciently smooth boundary $\partial D$
	 \begin{equation}\label{ba1}
	 \left\{
	 \begin{aligned}
	 u_{t}-\Delta u&=f(u)+g(t), \ \ t>0, \ x \in D, \\ 
	 u(x,t)&=0, \ \ t>0, \  x \in \partial{D},  \\
	 u(x,0)&=u_{0}(x), \ \ x \in D,
	 \end{aligned}
	 \right.
	 \end{equation}
	 when $g^\prime(t) \geq 0$ or $g(t)= \frac{k}{ \left|D \right|} \int_{\Omega}u_{t}(x,t)dx$, here $|D|$ is the volume of $D$ with $0 <k<1$. Chadam et al. \cite{chadam}  proved that under suitable conditions, the solution of \eqref{ba1} with $f(u)=0$ and $g=\int_{D}F(u)dx$ blows-up in a finite-time. M. Wang and Y. Wang \cite{wang1996} investigated the existence and non-existence of global positive solutions for a non-local reaction-diffusion equation of the type	 \begin{equation}\label{ba2}
	 \left\{
	 \begin{aligned}
	 u_{t}-d \Delta u &=\int_{D} u^{q}(t,x) dx -k u^{p}(t,x), \ \ x \in D, \ t>0,\\
	 u(x,0)&=u_{0}(x), \ \ x \in D,\\
	 \end{aligned}
	 \right. 
	 \end{equation}
	 subject to the boundary (Dirichlet or Neumann type) conditions 
	 \begin{align}\label{ba3}
	 u(t,x)=0, \  \mbox{ or  }\ \frac{\partial u}{\partial \upnu}=0, \ x \in \partial D, \ t>0,
	 \end{align}
where $D\subset  \mathbb{R}^d\ (d \geq 1)$ with suffciently smooth boundary $\partial D$. Here $\upnu$ is the unit outside normal to the boundary $\partial D$, the parameters $k>0, p,q \geq 1$ and $d>0$. Song in \cite{song} obtained the lower bounds for the finite-time blow-up of a non-local reaction-diffusion equation for \eqref{ba2}-\eqref{ba3} with $d=1.$  Li et al. \cite{cai}  proved that the solution exists globally or blow-up in finite-time to the reaction-diffusion system of non-local sources with appropriate hypotheses. 
%
	 
	 Since the above models are deterministic in nature, the dynamics is entirely determined by the initial data. This is obviously not the case with natural phenomena, where random environmental influences often play an important role. Many researchers have studied the blow-up results for non-linear stochastic partial differential equations (SPDEs) with the Dirichlet boundary conditions (cf. \cite{chow2009,doz2010,doz2013,doz2020}) etc.  Chow \cite{chow09,chow11} proved that the solution of the stochastic reaction-diffusion equation is positive and explodes in finite-time in the mean $L^p$-norm.  The stochastic counterpart of non-local problems have received less attention. Recently, Liang and Zhao \cite{liang}  addressed the existence and blow-up of positive solutions for a stochastic non-local reaction-diffusion equation of the form:
	 \begin{equation}
	 \left\{
	 \begin{aligned}
	 u_{t}-d_{1} \Delta u &=\int_{D} u^{q}(t,x) dx -k u^{p}(t,x)+\epsilon u(t,x)dB_{t}, \ \ x \in D, \ t>0,\\
	 u(t,x)&=0, \ \ t>0, \ x \in \partial D,\nonumber\\ 
	 u(0,x)&=u_{0}(x), \ \ x \in D, \nonumber\\
	 \end{aligned}
	 \right. 
	 \end{equation}
	 where 
	 $d_{1}, k, \epsilon>0$ and $p, q \geq 1$ are constants.
	 	 
	Inspired by the above facts, in the present paper, our aim is to prove the global existence of a weak solution for certain parameters  and obtain lower and upper bounds for the blow-up time for a more general class of stochastic non-local reaction-diffusion equations driven by fractional Brownian motion (fBm):  
	\begin{equation}\label{b1}
	\left\{
	\begin{aligned}
	du(t,x)&=\left[(\Delta+\gamma) u(t,x)+\int_{D}u^{q}(t,y)dy -ku^{p}(t,x)+\delta u^{m}(t,x)\int_{D}u^{n}(t,y)dy \right]dt\\
	&\quad+\eta u(t,x)dB^{H}(t),\\
	u(t,x)&=0, \ \ t>0, \ \ x\in \partial D, \\
	u(0,x)&=f(x) \geq 0, \ \ x\in D,\\
	\end{aligned}
	\right.
	\end{equation}
	where $D$ is a bounded domain in $\mathbb{R}^{d}\ (d\geq 1)$ with smooth boundary $\partial D$. Here $k>0, \gamma\geq 0, \delta \geq 0, \eta \geq 0$ and $p,q,n>1,\ m \geq 0$  with $m+n \geq q\geq p$. The initial data $f$ is  non-negative and bounded measurable  in the class $C^{2}$ which is not identically zero. Further, $\left\{ B^{H}(t) \right\}_{t \geq 0} $ is a standard one-dimensional fBm defined on a filtered probability space $\left( \Omega, \mathcal{F}, (\mathcal{F}_{t})_{t \geq 0}, \mathbb{P} \right)$ with Hurst parameter $\frac{1}{2} \leq H<1$. An fBm of Hurst parameter $H \in (0,1)$ is a centered Gaussian process $\left\{ B^{H}(t)\right\}_{t \geq 0}$, with covariance function 
	\begin{align*}
		R_{H}(t,s):=E[B^H(t)B^H(s)]= \frac{1}{2} \left( s^{2H}+t^{2H}-|t-s|^{2H} \right). 
	\end{align*}
	It is known that $B^{H}(\cdot)$  admits the so called \emph{Volterra representation} (Nualart \cite{nualart}),
	\begin{align}
		B^{H}(t):=\int_{0}^{t} K_{H}(t,s) dW(s),  \nonumber 
	\end{align}  
	where $\{W(t)\}_{t\geq 0}$ is a standard Brownian motion and the Volterra kernel $K_{H}(t,s)$ is defined by 
	\begin{align}
		K_{H}(t,s) = C_{H} \left[ \frac{t^{H-\frac{1}{2}}}{s^{H-\frac{1}{2}}}(t-s)^{H-\frac{1}{2}} -\left(H-\frac{1}{2}\right)\int_{s}^{t} \frac{u^{H-\frac{3}{2}}}{s^{H-\frac{1}{2}}}(u-s)^{H-\frac{1}{2}} du  \right], \ \ s \leq t, \nonumber 
	\end{align}
	where $C_{H}$ is a constant depending only on $H.$   When $H=\frac{1}{2}$, it is  the standard one-dimensional Brownian motion  $W(\cdot)$.  We closely follow the paper \cite{liang} to obtain our results. 
	
	
	The main goal of this article is to obtain  the existence and non-existence of a solution $u(\cdot, \cdot)$ of the equation \eqref{b1}. First, we estimate a lower bound for the finite-time blow-up of the solution of \eqref{b1} and by choosing a suitable initial data we obtain the upper bound for the finite-time blow-up of the equation \eqref{b1}. Further, we provide a sufficient condition for a weak solution $u$ of \eqref{b1} to exist globally and obtain the bounds for the probability of blow-up solution $u(\cdot,\cdot)$ of the equation \eqref{b1}. 
	
	The rest of the article is organized as follows: Section \ref{sep1} covers the preliminaries required for this work. Section \ref{sec2} is devoted to obtain an associated random PDE, by applying a random transformation (cf. \eqref{tr1}) to the equation \eqref{b1}. The result established in Theorem \ref{thm2.1} will be useful to discuss our further sections. The existence of random PDE is based on a transformation which converts a stochastic partial differential equation (SPDE) into an equivalent random PDE. Such type of transformation is available in the literature when the noise is either additive or linear multiplicative. The results obtained in Sections \ref{sec2}, \ref{sec3} and \ref{sec4} are outlined in the following table: 
	
\vspace{0.5 cm}
\begin{tabular}{|l|l|l|} \hline	
	Parameters&Theorem&Result\\
	\hline 
	$|D| \leq k$ \& &Theorem \ref{thm4.1} &The existence of a global solution $v(\cdot,\cdot)$ of the \\ $p=q=n, m=0$ & &equation \eqref{s1}.\\ \hline 
	& Theorem \ref{t2} &Lower bounds for the finite-time blow-up $\tau$  of  \\ &  & \eqref{b1} with any non-negative  initial value $f$.\\ 
	\cline{2-3}		$|D| > k$ \&  & Theorem \ref{mthm1} &The existence of a global solution $u(\cdot,\cdot)$ of the  \\ $m+n \geq q \geq p$& &equation\eqref{b1} under the condition \eqref{b2}. \\ \cline{2-3}	
	&Theorem \ref{thm5.1} &Upper bounds for $\tau$ with the initial data \\ 
	& & $f(x) \geq b\varphi(x),$ $x\in D$,  for any $b>1$ such that  \\ & & \eqref{B1} holds.\\
	\hline
\end{tabular}
\vspace{0.4 cm}
\noindent \\
	Further, we estimate the blow-up probability to the solution of \eqref{b1} for the case $m+n=q>p$, the initial value $f(x)\geq b \varphi(x),\ b>1$ and   $\frac{1}{2}<H<1$ by using  the Malliavin calculus and the method adopted in \cite{dung}. We consider the case $H=\frac{1}{2}$ in the last  section and estimate the upper and lower bounds for the blow-up time of the solution of the equation \eqref{b1}. Further, we establish the lower and upper bounds for the probability of the blow-up solution $u(\cdot,\cdot)$ of the equation \eqref{b1}.
          
   	\section{Preliminaries}\label{sep1}
	In this section, we first recall some known results and  we begin with some necessary function spaces to be used in the subsequent sections. Let $D$ be a bounded domain in $\mathbb{R}^{d}, d \geq 1$ with smooth boundary $\partial D,$  $\overline{D}$ be the  closure of $D$ and $|D|$ be the volume of $D$. We denote $L^{2}(D )=L^{2}(D ; \mathbb{R}),$ the usual $L^{2}$-real Hilbert space equipped with the  norm $\|u\|=\left(\int_D|u(x)|^2dx\right)^{1/2}$ and inner product $(u,v)=\int_{D}u(x)v(x)dx$ for all $u,v\in L^{2}(D )$.
	
Let the semigroup $\left\{ T_{t} \right\}_{t\geq 0}$  of bounded linear operators be defined by 
	\begin{eqnarray}
	T_{t}f(x)=\mathbb{E}\left[ f(X_{t}), \ t<\tau_{D}|X_{0}=x \right], \quad  x\in D,\nonumber
	\end{eqnarray}
	for all bounded and measurable functions $f : D\rightarrow \mathbb{R},$ where $\left\{ X_{t} \right\}_{t\geq 0}$ is the $d$-dimensional Brownian motion with variance parameter $2$, killed at the time $\tau_{D}$ at which it strike the boundary $\partial D.$ Moreover,   $\lambda_{1}>0$ is the first eigenvalue of the Dirichlet Laplacian operator $-\Delta$ on $D$ which satisfies 
	\begin{eqnarray} \label{a3}
	-\Delta \varphi(x)=\lambda_{1} \varphi(x), \ \ x\in D,
	\end{eqnarray}
where 	$\varphi$ is the corresponding eigenfunction which  is strictly positive on $D$  (see Corollary 3.3.7 in \cite{davies}) and $\varphi|_{\partial D} =0.$ Remember that $$T_{t}\varphi(x)=\exp\{{-\lambda_{1} t}\}\varphi(x), \ \ t\geq 0,\ x\in D.$$ We assume that $\varphi$ is normalized so that $\displaystyle\int_{D} \varphi(x)dx=1$. Let $\left\lbrace p_{t}(x,y)\right\rbrace_{t>0}$ be the transition kernel of $\left\lbrace T_{t}\right\rbrace_{t \geq 0}$ and one can refer to \cite{wang1992} for its properties. For any $\varphi\in C^\infty_0(D)$, we denote
	$$u(t,\varphi):=\int_{D}u(t,x) \varphi(x)dx,\ \ t\geq 0,\ x\in D.$$

	Let us recall the notion of weak solution of \eqref{b1}. Let $\tau < +\infty$ be a stopping time. A continuous  $\{\mathcal{F}_{t} \}_{t \geq 0}$-adapted random field $u=\left\lbrace u(t,x),\ t \geq 0,\ x \in D \right\rbrace $ is a \emph{weak solution} of \eqref{b1} on the interval $(0,\tau)$ provided 
	\begin{align}\label{c1} 
    u\left(t,\varphi\right)&=u\left(0,\varphi\right)+\int_{0}^{t}u\left(s,(\Delta+\gamma) \varphi\right)ds+\int_{0}^{t}\int_{D}u^{q}(s,y)dy \int_{D}\varphi(z)dz ds-k\int_{0}^{t}u^{p}(s,\varphi)ds\nonumber\\
    &\quad+\delta \int_{0}^{t}u^{m}\left(s,x\right)\int_{D}u^{n}(s,y) dy \int_{D} \varphi(z)dz ds +\eta \int_{0}^{t}u(s,\varphi) dB^{H}(s),\ \mathbb{P}\text{-a.s.}, 
	\end{align}
	holds for every $\varphi \in C^\infty_0(D).$ 
	For any bounded measurable initial data $f \geq 0,$ we say that $u=\left\lbrace u(t,x),\ t \geq 0,\ x \in D \right\rbrace $  is a local mild solution of the equation \eqref{b1}, if there exists a number $0<\tau=\tau(\omega)\leq \infty$ such that $u$ satisfies the integral equation (\cite[Chapter 4]{pazy}),
	\begin{align}
	u(t,x)&=\exp\left\lbrace \gamma t\right\rbrace T_{t}f(x)+\int_{0}^{t}  e^{\gamma (t-r)} T_{t-r}\Bigg[\int_{D}u^{q}(r,y)dy-ku^{p}(r,x)\nonumber\\
	&\qquad +\delta u^{m}(r,x)\int_{D} u^{n}(r,y)dy \Bigg] dr+\eta \int_{0}^{t}  e^{\gamma (t-r)} T_{t-r} u(r,x)  dB^{H}(r),\ \mathbb{P}\text{-a.s.}, \nonumber 
	\end{align}
	for each $0 \leq t<\tau$ and $x\in D$. The equivalence between weak and mild solutions of \eqref{b1} can be established in a similar way as in \cite[Proposition 3.7]{IGCR}. 
	
	The following result will be used in the sequel. 
    \begin{lemma}\label{l2} (\cite[Lemma 1]{doz2020}) Let $0<H<1$ and let $\left\{ B^{H}(t)\right\}_{t \geq 0} $ be a one-dimensional fractional Brownian motion with Hurst parameter $H$ defined on a probability space $\left( \Omega, \mathscr{F}, \mathbb{P} \right).$ If $\nu>0,$ then $$\mathbb{P}\left( \int_{0}^{\infty} e^{ B^{H}(s)-\nu s}   ds < \infty\right)=1.$$ If $\nu<0,$ then $$\mathbb{P}\left( \int_{0}^{\infty} e^{ B^{H}(s)-\nu s} ds = \infty\right)=1.$$  
    \end{lemma}
    
	\section{Random PDE}\label{sec2}
	In this section, we obtain a random PDE by making use of the random transformation 
	\begin{align}\label{tr1} 
	v(t,x) = \exp\{-\eta B^{H}(t)\}u(t,x), 
	\end{align} 
	for $\frac{1}{2} < H<1$ and $t \geq 0,\ x \in D$. The equation $\eqref{b1}$ is transformed into the following  random PDE:
	\begin{equation}\label{s1}
	\left\{
	\begin{aligned}
	\frac{\partial v(t,x)}{\partial t}&=(\Delta+\gamma)v(t,x)+ e^{\left( q-1\right)\eta B^{H}(t) }\int_{D}v^{q}(t,y)dy-ke^{\left( p-1\right) \eta B^{H}(t) }v^{p}(t,x)\\
	&\qquad +\delta e^{(m+n-1) \eta B^{H}(t)}v^{m}(t,x)\int_{D} v^{n}(t,y)dy,\\
	v(t,x)&=0, \ \ x \in \partial D,\ t>0, \\
	v(0,x)&=f(x), \ \  x \in D.  
	\end{aligned}
	\right.
	\end{equation}
    The following result provides an equivalence between the weak solutions of the stochastic PDE $\eqref{b1}$ and the random PDE \eqref{s1}. 
	\begin{theorem} \label{thm2.1}
	If $u(\cdot,\cdot)$ is a weak solution of $\eqref{b1}$. Then the function $v(\cdot,\cdot)$ defined by
	\begin{eqnarray}
	v(t,x) = \exp\{-\eta B^{H}(t)\}u(t,x), \ t \geq 0,\ x \in D, \nonumber 
	\end{eqnarray} 
	is a weak solution of the random PDE $\eqref{s1}$  and viceversa. 
	\begin{proof}
	By using Ito's formula (\cite[Lemma 2.7.1]{mis2008}), we have 
	\begin{eqnarray}
	e^{-\eta B^{H}(t)}=1- \eta \int_{0}^{t} e^{- \eta  B^{H}(s)}dB^{H}(s). \nonumber 
	\end{eqnarray}
	For any $\varphi\in C^\infty_0(D),$ the weak solution of $\eqref{b1}$ is given by
	\begin{align} 
	u\left(t,\varphi\right)&=u\left(0,\varphi\right)+\int_{0}^{t}u\left(s,(\Delta+\gamma) \varphi\right)ds+\int_{0}^{t}\int_{D}u^{q}(s,y)dy \int_{D}\varphi(z)dz ds-k\int_{0}^{t}u^{p}(s,\varphi)ds\nonumber\\
	&\quad+\delta \int_{0}^{t}u^{m}\left(s,x\right)\int_{D}u^{n}(s,y) dy \int_{D} \varphi(z)dz ds +\eta \int_{0}^{t}u(s,\varphi) dB^{H}(s).
	\end{align}
    By applying the integration by parts formula (\cite[Chapter 8]{klebaner}), we obtain
	\begin{align*}
	v(t,\varphi):&=\int_{D}v(t,x)\varphi(x)dx \nonumber\\
    &=v(0,\varphi)+\int_{0}^{t} e^{- \eta B^{H}(s)}du(s,\varphi)+\int_{0}^{t} u(s,\varphi)\left(-\eta e^{-\eta B^{H}(s)}dB^{H}(s)\right).\nonumber
	\end{align*} 
	Therefore,
	\begin{align}\label{a6}
	v(t,\varphi) &= v(0,\varphi)+\int_{0}^{t}v\left(s,(\Delta+\gamma) \varphi\right)ds+\int_{0}^{t}e^{(q-1) \eta B^{H}(s)}\int_{D}v^{q}(s,y)dy\int_{D}\varphi(z)dzds \nonumber\\
	&-k\int_{0}^{t}e^{(p-1) \eta B^{H}(s)}v^{p}(s, \varphi)ds +\delta\int_{0}^{t}e^{(m+n-1) \eta B^{H}(s)}v^{m}(s,x)\int_{D}v^{n}(s,y)dy\int_{D}\varphi(z)dzds.
	\end{align}
	The preceding equality means that $v(\cdot,\cdot)$ is a weak solution of $\eqref{s1}.$ The converse part follows from the fact that the change of variable is given by a homeomorphism  which transforms one random dynamical system into an another equivalent one.   
	\end{proof}
	\end{theorem}


	This equation \eqref{b1} is understood in the pathwise (or trajectory wise) sense and the existence and uniqueness of local solution is clear by a standard theory of parabolic type.  It is also clear that the solution can be extended in $t$ direction, as long as the $L^{\infty}$ norm of the solution remains finite. Moreover,  the classical results for parabolic PDE can be applied to show existence and uniqueness of solution $v \in \mathrm{C}\left([0,\tau) ; L^{p}(D ; \mathbb{R})\right)$ up to an eventual blowup (see  \cite[Chapter 7]{friedman1964}  and \cite{liang}).

	  Let us denote 
	\begin{align}
	B^{H}_{\ast}(t) = \sup_{0 \leq s \leq t} |B^{H}_{s}|,\ \mbox{for each}\ t\geq 0. \nonumber 
	\end{align}
	\begin{lemma}(see \cite[Lemma 2.1]{dung2})
	 Let $Z$ be a centered random variable in $\mathbb{D}^{1,2}$ (cf. Sec. \ref{prb1} for more details). Assume that there exists a non-random constant $M_{0}$ such that $$ \int_{0}^{T} (D_{r} Z)^{2} dr \leq M_{0}^{2},\ \mbox{a.s.}$$
	Then the following estimate for tail probabilities holds:
	$$\mathbb{P}( |Z| \geq x ) \leq 2 e^{-\frac{x^{2}}{2M_{0}^{2}}},\ x>0.$$
   \end{lemma}
  \noindent  Therefore (see \eqref{520} below) $$ \mathbb{P}(B_{\ast}^{H}(t)< \infty)=1-\lim_{x \rightarrow \infty} \mathbb{P}(B_{\ast}^{H}(t)> x)=1,\ t>0.$$
   Let us denote $N_{t}=\{\omega \in \Omega : B^{H}_{\ast}(t)= \infty\},$ for every $t>0.$ Then, we have $\mathbb{P}(N_{t})=0,$ for every $t>0.$ If we take $t=1,2,\ldots,$ it is clear that $N_{t} \subset N_{m},$ for $t \geq m.$ Define $N=\displaystyle\lim_{m \rightarrow \infty} N_{m},$ we have $\mathbb{P}(N)=\displaystyle\lim_{m\to \infty}\mathbb{P}(N_{m})=0.$ Therefore, for all $\omega \in N, B^{H}_{\ast}(t;\omega)<\infty,$ for all $t>0.$ From the above observation, without loss of generality, we can assume that for all $\omega \in \Omega,$
   \begin{align} \label{fb1}
   B^{H}_{\ast}(t;\omega)<\infty,\ \mbox{for all}\ t>0.
   \end{align} 
   By Theorem \ref{thm2.1} and \eqref{fb1},  the global existence and finite-time blow-up of $u(\cdot,\cdot)$ is guarenteed by considering the properties of $v(\cdot,\cdot)$. So in this paper, we mainly consider the random partial differential equation \eqref{s1} for further investigation.

	From classical results in semigroup theory, for any bounded measurable initial data $f \geq 0,$ there exists a unique local mild solution $v(\cdot,\cdot)$ of the random PDE \eqref{s1}, if there exists a number $0<\tau=\tau(\omega)\leq \infty$ such that $v(t,x)$ satisfies the integral equation \cite[Chapter 4]{pazy}
	\begin{align}\label{e3}
	v(t,x)&=\exp\left\lbrace \gamma t\right\rbrace T_{t}f(x)+\int_{0}^{t}  e^{\gamma (t-r)} T_{t-r}\Bigg[e^{\left( q-1\right) \eta B^{H}(r) }\int_{D}v^{q}(r,y)dy\nonumber\\
	&\qquad-ke^{\left( p-1\right) \eta B^{H}(r) }v^{p}(r,x) +\delta e^{(m+n-1) \eta B^{H}(r)}v^{m}(r,x)\int_{D} v^{n}(r,y)dy \Bigg] dr,
	\end{align}
    for each $0 \leq t<\tau$ and $x\in D.$ 
    
    Let $\tau$ be the blow-up time of the equation \eqref{s1} with the initial value $f$. Due to Theorem \ref{thm2.1} and a.s continuity of $B^H(\cdot)$, $\tau$ is also the blow-up time for the equation \eqref{b1} as well (\cite[Corollary 1]{dozfrac2010}).  
    
    The following comparison result  is a consequence of  \cite[Lemma 2.1]{li2003} and  \cite[Lemma 2.1]{wang1996} (for more details see  \cite[Lemma 1]{pao}).
    \begin{lemma}\label{l1}
    	Suppose that $w(t,x) \in C^{1,2}(D_{T}) \cap C(\overline{D}_{T}),$ where $D_{T}=(0,T] \times D$ and satisfies 
    	\begin{equation}\label{a1}
    	 \left\{
    	 \begin{aligned}
    	w_{t}(t,x)-d\Delta w(t,x) &\geq c_{1}(t,x)w(t,x)+c_{3}(t,x)\int_{D}c_{2}(t,x)w(t,x) dx \\
    	&\quad+\int_{D} c_{4}(t,x)w(t,x)dx,\ x \in D,\ 0<t \leq T,\\
    	w(t,x) &\geq 0,\ x \in \partial D,\ 0\leq t\leq T,\\
    	w(0,x) &\geq 0,\ x \in D, 
    	 \end{aligned}
    	 \right.
    	\end{equation}
    	where $c_{1}(t,x), c_{2}(t,x), c_{3}(t,x)$ and $c_{4}(t,x)$ are bounded functions and $c_{2}(t,x) \geq 0,\ c_{3}(t,x) \geq 0,\ c_{4}(t,x) \geq 0$ and $d > 0$ in $D_{T}$. Then $w(t,x) \geq 0$ on $\overline{D}_{T}.$ Moreover, $w(t,x)>0$ in $\overline{D}_{T}$ if $w(0,x)$ is not identiaclly zero. 
    \end{lemma}
\begin{proof}
	Let $\widetilde{c}_{i}$ be the least upper bound of $c_{i}(t,x)$ in $D_{T}, i=1,2,3,4.$ Let $v(t,x)=e^{-\gamma_{1} t}w(t,x)$ for some $\gamma_{1} > \widetilde{c}_{1}+(\widetilde{c}_{2}\widetilde{c}_{3}+\widetilde{c}_{4})|D|.$ Therefore from \eqref{a1}, we have
	\begin{align}\label{a2}
	Lv :&= v_{t}(t,x)-d\Delta v(t,x) +(\gamma_{1}- c_{1}(t,x))v(t,x) \nonumber\\&\geq e^{-\gamma_{1} t}\left[c_{3}(t,x)\int_{D}c_{2}(t,x)v(t,x) dx+\int_{D} c_{4}(t,x)v(t,x)dx\right], \ x \in D,\ 0<t\leq T,  \\
	v(t,x) &\geq 0,\ x \in \partial D,\ 0<t\leq T,\nonumber \\
	v(0,x) &\geq 0,\ x \in D.\nonumber
	\end{align}
	We need to show that $v(t,x) \geq 0$ on $\overline{D}_{T}.$ Suppose that $v$ has a negative minimum for some $(t_{0},x_{0}) \in \overline{D}_{T}$. If $x_{0} \in \partial D,$ then $v(t_{0},x_{0}) \geq 0$ for $0<t_{0}\leq T,$ which is a contradiction, since $v(t_{0},x_{0}) < 0$. Hence $(t_{0},x_{0}) \in D_{T},$  and it is immediate that  $v_{t}(t_{0}, x_{0}) = 0,$ for $0<t_0<T$ and $v_{t}(t_{0}, x_{0}) \leq 0,$ for $t_0=T$,   $\Delta v(t_{0},x_{0}) \geq 0$ and 
	\begin{align}
	Lv(t_{0},x_{0}) \leq (\gamma_{1}-c_{1}(t_{0},x_{0}))v(t_{0},x_{0}). \nonumber 
	\end{align} 
	From \eqref{a2}, we have 
	\begin{align}
&e^{-\gamma_{1} t_0}\left[	c_{3}(t_{0},x_{0})\int_{D}c_{2}(t_{0},x)v(t_{0},x) dx+\int_{D} c_{4}(t_{0},x)v(t_{0},x)dx\right] \leq (\gamma_{1}-c_{1}(t_{0},x_{0}))v(t_{0},x_{0}), \nonumber\\
&\Rightarrow 	(\widetilde{c}_{3}\widetilde{c}_{2} |D|+ \widetilde{c}_{4}|D|)v(t_{0},x_{0}) \leq (\gamma_{1}-\widetilde{c}_{1})v(t_{0},x_{0}), \nonumber\\
&\Rightarrow 	\gamma_{1} \leq \widetilde{c}_{1}+(\widetilde{c}_{3}\widetilde{c}_{2} + \widetilde{c}_{4})|D|, \nonumber
	\end{align}
	which is a contradiction, since $\gamma_{1} > \widetilde{c}_{1}+(\widetilde{c}_{2}\widetilde{c}_{3}+\widetilde{c}_{4})|D|.$ Hence $v(t,x) \geq 0$ on $\overline{D}_{T}$  implies that $w(t,x)=e^{\gamma_{1} t}v(t,x) \geq 0$ in $\overline{D}_{T}.$ 
\end{proof}
  
    Let $N>0$ be a constant and define the following stopping time:
    $$\tau_{N}(\omega)=\inf \left\lbrace t>0:\ |B^{H}_{t}(\omega)|\geq N\right\rbrace.$$
   Clearly, $$\left\lbrace \omega \in \Omega :\ \tau_{N}(\omega) \leq t \right\rbrace=\left\lbrace \omega \in \Omega:\ B^{H}_{\ast}(t) \geq N \right\rbrace.$$
    Let $z(t,x)$ be the solution of 
    \begin{equation*}
    \left\{
    \begin{aligned}
    \frac{\partial z(t,x)}{\partial t}&=(\Delta+\gamma)z(t,x)+ e^{\left( q-1\right) \eta B^{H}(t) }\int_{D}|z(t,y)|^{q-1}z(t,y)dy-ke^{\left( p-1\right) \eta B^{H}(t) }|z(t,x)|^{p-1}z(t,x)\\
    &\qquad +\delta e^{(m+n-1) \eta B^{H}(t)}|z(t,x)|^{m}\int_{D} |z(t,y)|^{n-1}z(t,y)dy,\ (x,t) \in D \times (0,T \wedge \tau_{N}],\\
    z(t,x)&=0, \ \ (x,t) \in \partial D \times (0,T \wedge \tau_{N}], \\
    z(0,x)&=f(x), \ x \in D.  
    \end{aligned}
    \right.
    \end{equation*} 
    Here $f(x) \geq 0,\ c_{1}(t,x)=\gamma-ke^{(p-1) \eta B^{H}(t)}|z(t,x)|^{p-1},\ c_{2}(t,x)=\delta e^{(m+n-1) \eta B^{H}(t)}|z(t,x)|^{n-1} \geq 0,\ c_{3}(t,x)=|z(t,x)|^{m}\geq 0,\ c_{4}(t,x)= e^{\left( q-1\right) \eta B^{H}(t)}|z(t,y)|^{q-1}\geq 0$ and $c_{1}(t,x), c_{2}(t,x), c_{3}(t,x)$ and $c_{4}(t,x)$ are bounded in $D \times (0,T \wedge \tau_{N}].$ By using Lemma \ref{l1}, we have $z(t,x) \geq 0$ and hence, $z(t,x)$ is the solution of \eqref{s1}. Moreover, by uniqueness, we have $v(t,x)=z(t,x)\geq 0.$
    Again by using Lemma \ref{l1}, we obtain  that if $f(x)$ is not identically zero, then $v(t,x)>0$ in $D \times (0,T \wedge \tau_{N}]$.

    \begin{Rem}
    If the initial data $f$ is non-negative, then  the weak solution $u(\cdot,\cdot)$ and $v(\cdot,\cdot)$, of \eqref{b1} and \eqref{s1} respectively are non-negative as well. 
    \end{Rem}
    
    Using the increasing property of the function $x^{i}$ for $x \geq 0$, $i \in \left\lbrace p,q,m,n\right\rbrace$ with $p,q,n>1$ and $m \geq 0,$ as a consequence of Lemma \ref{l1}, we have the following comparison principle:
    \begin{proposition}\label{p2}
    Let $v(\cdot,\cdot)$ be the solution of poblem \eqref{s1} and the positive function $V \in C^{2,1}(D_{T}) \cap C(\overline{D}_{T}),$ where $D_{T}=D \times (0,T \wedge \tau_{N}]$ satisfy
    \begin{equation*}
    \left\{
    \begin{aligned}
    \frac{\partial V(t,x)}{\partial t}&\geq (\leq)(\Delta+\gamma)V(t,x)+ e^{\left( q-1\right) \eta B^{H}(t) }\int_{D}V^{q}(t,y)dy-ke^{\left( p-1\right) \eta B^{H}(t) }V^{p}(t,x)\\
    &\qquad +\delta e^{(m+n-1) \eta B^{H}(t)}V^{m}(t,x)\int_{D} V^{n}(t,y)dy,\ (x,t) \in D \times (0,T \wedge \tau_{N}], \\
    V(t,x)&\geq (\leq)0, \ \ (x,t) \in \partial D \times (0,T \wedge \tau_{N}], \\
    V(0,x)&\geq (\leq)f(x), \ x \in D.  
    \end{aligned}
    \right.
    \end{equation*}
    Then 
    \begin{align}
    V(t,x) \geq (\leq)v(t,x)\ \mbox{on}\ \overline{D} \times [0,T\wedge \tau_{N}].\nonumber 
    \end{align}      
    \end{proposition}
	
	The following theorem establishes the existence of a  global solution $v(\cdot,\cdot)$ of \eqref{s1} with $|D| \leq k.$
	\begin{theorem}\label{thm4.1}
	Let $v(\cdot,\cdot)$ be a solution of \eqref{s1} and $\frac{1}{2}<H<1$. If $p=q=n=\beta\ (say),\ m=0$  and $|D| \leq k,$ then there exist $\sigma_{0}>0$ and $M_{1}>0$ such that  $$v(t,x) \leq M_{1} e^{-\sigma_{0} t +(\beta-1) \eta B^{H}_{\ast}(t)},\ \mbox{for all}\ (t,x) \in [0, \infty) \times \overline{D}.$$ 
	\end{theorem} 
	The proof follows from  \cite[Theorem 3.3]{liang}.

	\noindent
	Our next aim is to obtain random times $\tau_{*}$ and $\tau^*$  such that $0\leq \tau_*\leq \tau\leq \tau^*.$ 
	\section{A Lower Bound for $\tau$}\label{sec3}
	This section is devoted to obtain the lower bounds $\tau_{\ast}$ to the blow-up time $\tau$ such that $\tau_{\ast} \leq \tau$ of \eqref{b1} with $\frac{1}{2}<H<1$. Also, we provide a sufficient condition for the existence of a global mild solution of	the equation \eqref{b1} for $|D|>k$. 
	
	Let us first consider the problem:
    \begin{equation}\label{re3}
    \left\{
    \begin{aligned}
    \frac{\partial w(t,x)}{\partial t}&=(\Delta+\gamma)w(t,x)+ e^{\left( q-1\right)\eta B^{H}(t) }\int_{D}w^{q}(t,y)dy+\delta e^{(m+n-1) \eta B^{H}(t)}w^{m}(t,x)\int_{D} w^{n}(t,y)dy,\\
    w(t,x)&=0, \ \ x \in \partial D,\ t>0, \\
    w(0,x)&=f(x), \ \  x \in D.  
    \end{aligned}
    \right.
    \end{equation}
    Then the mild solution (see \cite[Chapter 4]{pazy}), $w(\cdot,\cdot)$ of \eqref{re3} satisfies the following integral equation:
    \begin{align}
    w(t,x)=e^{\gamma t} T_{t}f(x)+&\int_{0}^{t}  e^{\gamma (t-r)} T_{t-r}\Bigg[e^{\left( q-1\right) \eta B^{H}(r) }\int_{D}w^{q}(r,y)dy\nonumber\\
    &\qquad+\delta e^{(m+n-1) \eta B^{H}(r)}w^{m}(r,x)\int_{D} w^{n}(r,y)dy \Bigg] dr,
    \end{align}
    for each $x \in D$ and $0 \leq t<\tau.$ The following result provides a lower bound for the finite-time blow-up of the solution of \eqref{b1}.

	\begin{theorem}\label{t2} Assume that $n,p,q>1,\ m\geq 0$ with $ m+n \geq q \geq p >1$,  $f$ is a non-negative bounded function and $|D|>k$. Let $\tau_{\ast}$ be  given by
	\begin{align}\label{t11}
	\tau_{\ast} = \inf \Bigg\{ t\geq 0 :&\int_{0}^{t}\left( e^{\left( q-1\right) \eta B^{H}(r)} \vee e^{(m+n-1) \eta B^{H}(r)} \right)\| e^{\gamma r }T_{r}\|_{\infty}^{m+n-1}dr\nonumber\\ &	\qquad\geq\frac{1}{2M(m+n+q-1) \|f\|_{\infty}^{m+n-1}} \Bigg\},
	\end{align}
    where $M = \max\left\lbrace |D|,\delta|D| \right\rbrace\ \mbox{and}\ \| f \|_{\infty}:=\displaystyle\sup_{x \in D} f(x).$  	Then $\tau_{\ast}\leq\tau$. 
      \end{theorem}

	\begin{proof}
	Let $w(\cdot,\cdot)$ solve the equation \eqref{re3}. Then, we have
	\begin{align}
	w(t,x)&=e^{\gamma t} T_{t}f(x)+\int_{0}^{t}  e^{\left( q-1\right) \eta B^{H}(r)+\gamma (t-r)} T_{t-r}\left( \int_{D}w^{q}(r,y)dy\right)dr \nonumber\\
	&\qquad+\delta \int_{0}^{t}  e^{(m+n-1) \eta B^{H}(r)+\gamma (t-r)} T_{t-r}\left(w^{m}(r,x)\int_{D} w^{n}(r,y)dy \right)  dr. \nonumber 
	\end{align}
	For all $x\in D,  \ t\geq 0,$ let us define the operator $\mathcal{H}$ as follows:
	\begin{align}
	\mathcal{H}v(t,x)&=e^{\gamma t} T_{t}f(x)+\int_{0}^{t}  e^{\left( q-1\right) \eta B^{H}(r)+\gamma (t-r)} T_{t-r}\left( \int_{D}v^{q}(r,y)dy\right)dr \nonumber\\  
	&\qquad+\delta\int_{0}^{t}  e^{(m+n-1) \eta B^{H}(r)+\gamma (t-r)} T_{t-r}\left(v^{m}(r,x)\int_{D} v^{n}(r,y)dy \right)  dr, \nonumber
	\end{align}
	where $v$ is any non-negative, bounded and measurable function.
    First, we shall prove that 
	\begin{eqnarray}
    \mathcal{H}w(t,x)=w(t,x),\ \ x\in D, \ \ 0 \leq t<\tau_{*},\nonumber
	\end{eqnarray}
	for some non-negative bounded and measurable function $w$.  Moreover, on the set $t < \tau_{*},$ we set
	\begin{align}
	\mathscr{G}(t)&=\Bigg[ 1-2(m+n+q-1)M \int_{0}^{t}\left( e^{\left( q-1\right) \eta B^{H}(r)} \vee e^{(m+n-1) \eta B^{H}(r)} \right)\nonumber\\
	&\qquad\qquad\qquad\times\| e^{\gamma r }T_{r}\|_{\infty}^{m+n-1}\|f\|_{\infty}^{m+n-1} dr \Bigg]^{-\frac{1}{m+n+q-1}}.\nonumber
	\end{align}
    Then, it can be easily seen that 
    \begin{equation}
    	\left\{
	\begin{aligned} 
	\frac{d \mathscr{G}(t)}{dt}&=2M \left( e^{\left( q-1\right) \eta B^{H}(t)} \vee e^{(m+n-1) \eta B^{H}(t)} \right) \| e^{\gamma t }T_{t}\|_{\infty}^{m+n-1}\|f\|_{\infty}^{m+n-1} \mathscr{G}^{m+n+q}(t), \nonumber\\ 
	\mathscr{G}(0)&=1, \nonumber
	\end{aligned}
	\right.
	\end{equation}
	so that 
	\begin{align}
	\mathscr{G}(t)&=1+2M\int_{0}^{t} \left( e^{\left(q-1\right) \eta B^{H}(r)} \vee e^{(m+n-1) \eta B^{H}(r)} \right)\| e^{\gamma r }T_{r}\|_{\infty}^{m+n-1}\|f\|_{\infty}^{m+n-1} \mathscr{G}^{m+n+q}(r) dr. \nonumber
	\end{align}
	Let us choose $v\geq 0$ such that $$v(t,x)\leq e^{\gamma t }T_{t}\|f\|_{\infty}\mathscr{G}(t),$$ for $x\in D$ and $t<\tau_{*}.$ Then $e^{\gamma t}T_{t}\|f\|_{\infty}\leq \mathcal{H} v(t,x)$ and
	\begin{align}
    \mathcal{H}v(t,x)&=e^{\gamma t} T_{t}f(x)+\int_{0}^{t}  e^{\left( q-1\right) \eta B^{H}(r)+\gamma (t-r)} T_{t-r}\left( \int_{D}v^{q}(r,y)dy\right) dr \nonumber\\ 
	&\quad+\delta\int_{0}^{t}  e^{(m+n-1) \eta B^{H}(r)+\gamma (t-r)} T_{t-r}\left(v^{m}(r,x)\int_{D} v^{n}(r,y)dy \right)  dr\nonumber\\
	& \leq e^{\gamma t} T_{t}f(x)+\int_{0}^{t}  e^{\left( q-1\right) \eta B^{H}(r)+\gamma (t-r)} T_{t-r}\left( \int_{D}\left(e^{\gamma r }T_{r}\|f\|_{\infty}\mathscr{G}(r)\right)^{q}dy\right)dr \nonumber\\ 
	&\quad+\delta\int_{0}^{t}  e^{(m+n-1) \eta B^{H}(r)+\gamma (t-r)} T_{t-r}\left(\left( e^{\gamma r }T_{r}\|f\|_{\infty}\mathscr{G}(r)\right)^{m}\int_{D} \left( e^{\gamma r }T_{r}\|f\|_{\infty}\mathscr{G}(r)\right)^{n}dy \right)  dr\nonumber\\
	&\leq e^{\gamma t} T_{t}\|f\|_{\infty}+\int_{0}^{t}  e^{\left( q-1\right) \eta B^{H}(r)+\gamma (t-r)} \|e^{\gamma r }T_{r}\|_{\infty}^{q-1} \|f\|_{\infty}^{q-1}T_{t-r}(e^{\gamma r }T_{r}\|f\|_{\infty}) \mathscr{G}^{q}(r)|D|dr \nonumber\\
	&\quad+\delta\int_{0}^{t}  e^{(m+n-1) \eta B^{H}(r)+\gamma (t-r)} T_{t-r}\left( e^{\gamma r }T_{r}\|f\|_{\infty}\right) \| e^{\gamma r }T_{r}\|_{\infty}^{n+m-1}\|f\|_{\infty}^{m+n-1}\mathscr{G}^{m+n}(r)|D|   dr\nonumber\\
   	&= e^{\gamma t} T_{t}\|f\|_{\infty}\Bigg[1+|D|\int_{0}^{t}e^{\left( q-1\right) \eta B^{H}(r)} \|e^{\gamma r }T_{r}\|_{\infty}^{q-1} \|f\|_{\infty}^{q-1} \mathscr{G}^{q}(r)dr \nonumber\\ 
   	&\quad+\delta|D|\int_{0}^{t}  e^{(m+n-1) \eta B^{H}(r)} \| e^{\gamma r }T_{r}\|_{\infty}^{m+n-1}\|f\|_{\infty}^{m+n-1}\mathscr{G}^{m+n}(r)dr \Bigg]\nonumber\\
   	&\leq e^{\gamma t} T_{t}\|f\|_{\infty}\Bigg[1+M\int_{0}^{t} \left( e^{\left( q-1\right) \eta B^{H}(r)} \vee e^{(m+n-1) \eta B^{H}(r)} \right)\nonumber\\
   	&\quad\times \left( \|e^{\gamma r }T_{r}\|_{\infty}^{q-1} \|f\|_{\infty}^{q-1}+\| e^{\gamma r }T_{r}\|_{\infty}^{m+n-1}\|f\|_{\infty}^{m+n-1}\right) \left(  \mathscr{G}^{q}(r)+\mathscr{G}^{m+n}(r)\right) dr\Bigg] \nonumber\\
   	&\leq e^{\gamma t} T_{t}\|f\|_{\infty}\Bigg[1+2M\int_{0}^{t} \left( e^{\left(q-1\right) \eta B^{H}(r)} \vee e^{(m+n-1) \eta B^{H}(r)} \right)\nonumber\\
   	&\quad\times\| e^{\gamma r }T_{r}\|_{\infty}^{m+n-1}\|f\|_{\infty}^{m+n-1} \mathscr{G}^{m+n+q}(r) dr\Bigg] \nonumber\\  
   	&=e^{\gamma t}T_{t}\|f\|_{\infty} \mathscr{G}(t), \nonumber 
	\end{align}
	where $M = \max\left\lbrace |D|,\delta|D| \right\rbrace$. Thus, we have
	$$e^{\gamma t}T_{t}f(x)\leq \mathcal{H} v(t,x)  \leq e^{\gamma t} T_{t}\|f\|_{\infty}\mathscr{G}(t).$$
	For $x\in D,\ 0\leq t\leq \tau_{*},$    let us take,  $$u^{(0)}(t,x)=e^{\gamma t}T_{t}f(x)
	\mbox{ and }	u^{(n)}(t,x)=\mathcal{H} u^{(n-1)}(t,x).$$

	Our aim is to show that the sequence $\{ u^{(n)} \}_{n\in\mathbb{N}}$ of functions is monotonically increasing. Note that 
	\begin{align}
	u^{(0)}(t,x)&\leq e^{\gamma t} T_{t}f(x)+\int_{0}^{t}  e^{\left( q-1\right) \eta B^{H}(r)+\gamma (t-r)} T_{t-r}\left( \int_{D}(u^{(0)})^{q}(r,y)dy\right)dr \nonumber\\
	&\quad+\delta\int_{0}^{t}  e^{(m+n-1) \eta B^{H}(r)+\gamma (t-r)} T_{t-r}\left((u^{(0)})^{m}(r,x)\int_{D} (u^{(0)})^{n}(r,y)dy \right)  dr\nonumber\\
	&= \mathcal{H}u^{(0)}(t,x) = u^{(1)}(t,x).\nonumber
	\end{align}
	Now assume that $u^{(n)}\geq u^{(n-1)},$ for some $n \geq 1.$  Then the monotonicity of $\mathcal{H}$ leads to the inequality
	\begin{eqnarray}
	u^{(n+1)}(t,x)=\mathcal{H}  u^{(n)}(t,x)\geq \mathcal{H} u^{(n-1)}(t,x)=u^{(n)}(t,x),\nonumber
	\end{eqnarray}
	and therefore the limit 
	\begin{eqnarray*}
	\lim_{n\rightarrow \infty}u^{(n)}(t,x)=w(t,x),
	\end{eqnarray*}
	exists for $x\in D$ and $0 \leq t< \tau_{*}.$ As a result of the monotone convergence theorem, we obtain 
	\begin{eqnarray}
	w(t,x)=\mathcal{H}w(t,x).\nonumber
	\end{eqnarray}
	Moreover,  we have $$\mathcal{H} w (t,x)\leq e^{\gamma t} T_{t}\|f\|_{\infty} \mathscr{G}(t),$$
	so that
	\begin{align}\label{inq1}
	0 \leq w(t,x) \leq& e^{\gamma t}T_{t}\|f\|_{\infty}\Bigg[ 1-2(m+n+q-1)M \int_{0}^{t}\left( e^{\left(q-1\right) \eta B^{H}(r)} \vee e^{(m+n-1) \eta B^{H}(r)} \right)\nonumber\\
	&\qquad\times\| e^{\gamma r }T_{r}\|_{\infty}^{m+n-1}\|f\|_{\infty}^{m+n-1}  dr \Bigg]^{-\frac{1}{m+n+q-1}}. 	
	\end{align}
	From \eqref{re3}, we infer that 
	\begin{align}
	\frac{\partial w(t,x)}{\partial t}&\geq (\Delta+\gamma)w(t,x)+ e^{\left( q-1\right)\eta B^{H}(t) }\int_{D}w^{q}(t,y)dy\nonumber\\
	&\quad-ke^{\left( p-1\right) \eta B^{H}(t) }w^{p}(t,x)+\delta e^{(m+n-1) \eta B^{H}(t)}w^{m}(t,x)\int_{D} w^{n}(t,y)dy.\nonumber 
	\end{align}
	By Proposition \ref{p2}, we deduce that $0 \leq v(t,x) \leq w(t,x)$, for each $x \in D$ and $t\geq 0$, where $v(\cdot,\cdot)$ is the unique solution of \eqref{s1}. Therefore from \eqref{inq1}, we get
	\begin{align}
	0 \leq v(t,x)& \leq e^{\gamma t}T_{t}\|f\|_{\infty}\Bigg[ 1-2(m+n+q-1)M \int_{0}^{t}\left( e^{\left(q-1\right) \eta B^{H}(r)} \vee e^{(m+n-1) \eta B^{H}(r)} \right)\nonumber\\
	&\qquad\times\| e^{\gamma r }T_{r}\|_{\infty}^{m+n-1}\|f\|_{\infty}^{m+n-1}  dr \Bigg]^{-\frac{1}{m+n+q-1}}, \nonumber 	
	\end{align}	
	which completes the proof of the theorem.
	\end{proof}
  
	
	The following results show that the solution $u(\cdot, \cdot)$ of the equation \eqref{b1} exists globally.	
	\begin{theorem}\label{mthm1} 
	Assume that $n,p,q>1,\ m \geq 0$ with $m+n \geq q \geq p.$ If $f$ is a non-negaive bounded function, $|D|>k$ and the inequality 
	\begin{align} \label{b2}
	\Bigg\{2(m+n+q-1)M &\int_{0}^{t}\left( e^{\left(q-1\right) \eta B^{H}(r)} \vee e^{(m+n-1) \eta B^{H}(r)} \right)\| e^{\gamma r }T_{r}\|_{\infty}^{m+n-1}\|f\|_{\infty}^{m+n-1}  dr\Bigg\}<1,
	\end{align}  
    holds, then the weak solution  $v(t,x)$ of \eqref{s1} exists globally for all $(t,x) \in [0,\infty) \times D.$ Moreover, for all $(t,x) \in [0,\infty) \times D$
	\begin{align*}
	0 \leq v(t,x) &\leq e^{\gamma t}T_{t}\|f\|_{\infty}\Bigg[ 1-2(m+n+q-1)M \int_{0}^{t}\left( e^{\left(q-1\right) \eta B^{H}(r)} \vee e^{(m+n-1) \eta B^{H}(r)} \right)\nonumber\\
	&\qquad\times\| e^{\gamma r }T_{r}\|_{\infty}^{m+n-1}\|f\|_{\infty}^{m+n-1}  dr \Bigg]^{-\frac{1}{m+n+q-1}}, \nonumber
	\end{align*} 
	where $M = \max\left\lbrace |D|,\delta|D| \right\rbrace.$
	\end{theorem}
	The proof of Theorem \ref{mthm1} immediately follows from Theorem \ref{t2}.
	
	The following thoerem gives sharp bounds for $\left\lbrace p_{t}(x,y), \ t > 0 \right\rbrace$ (see \cite[Theorem 1.1]{wang1992}).
	\begin{theorem}\label{thm4.4}
	(\cite{wang1992}). Let $\varphi$ be the first (normalized) Dirichlet eigenfunction of \eqref{a3}  on a connected bounded $C^{1, \alpha}$-domain $D \subset \mathbb{R}^{d}$, where $\alpha>0$ and $d \geq 1$, and let $\{p_{t}(x,y), \ t > 0 \}$ be the corresponding Dirichlet heat kernel. Then, there exists a constant $c > 0$ such that 
	\begin{align}\label{c2}
	\max \left\lbrace 1,\frac{1}{c}t^{-\frac{(d+2)}{2}} \right\rbrace \leq e^{\lambda_{1}t}\sup_{x,y \in D} \frac{p_{t}(x,y)}{\varphi(x) \varphi(y)} \leq 1+c(1 \wedge t)^{-\frac{(d+2)}{2}}e^{-(\lambda_{2}-\lambda_{1})t}, \ t>0,   
	\end{align}	
	where $\lambda_{1},\lambda_{2}$ are the first two Dirichlet eigenvalues of  the Laplacian with $\lambda_{2}>\lambda_{1}$. This estimate is sharp for both short and long times.
	\end{theorem}
	The following result establishes the existence of a global weak solution to the problem \eqref{b1} under the assumption that $f(x) \leq C_{0} \varphi(x)$, for each $x \in D$ and $C_{0}>0$ by using Theorem \ref{thm4.4}.
	
	\begin{theorem}\label{ex1}
	Let $D$ be a bounded domain $D \subset \mathbb{R}^{d}\  (d \geq 1)$ with smooth boundary. Assume that $n,q>1,\ m \geq 0$ along with the conditions $m+n \geq q$ and $|D|>k$. If the initial data $f(x) \leq C_{0} \varphi(x)$  satisfies
	\begin{align} 
    \mathcal{K}\int_{0}^{\infty} \left( e^{\left(q-1\right) \eta B^{H}(r)} \vee e^{(m+n-1) \eta B^{H}(r)} \right) e^{-\left( \lambda_{1}-\gamma\right)(m+n-1) r} dr<1,\nonumber
	\end{align}  
	where $C_{0}>0$, $\mathcal{K}=2M(m+n+q-1)\left(C_{0}\left\| \varphi \right\|_{\infty}^{2}(1+c) \right)^{m+n-1}$ and $M = \max\left\lbrace |D|,\delta|D| \right\rbrace,$
	then the solution $v(\cdot,\cdot)$ of \eqref{s1} exists globally. 
	\end{theorem}
	\begin{proof}
	For any $f(x) \geq 0, x \in D$ and any $t>0$ 
	\begin{align}
	T_{t}f(x)&=\int_{D} p_{t}(x,y)f(y)dy =\int_{D} e^{\lambda_{1}t}\frac{p_{t}(x,y)}{\varphi(x) \varphi(y)}e^{-\lambda_{1}t}\varphi(x) \varphi(y) f(y)dy \nonumber\\
	&\leq \left\| f \right\|_{\infty} \left\| \varphi \right\|_{\infty} \int_{D} e^{\lambda_{1}t}\frac{p_{t}(x,y)}{\varphi(x) \varphi(y)}e^{-\lambda_{1}t} \varphi(y) dy. \nonumber
	\end{align} 
	By using \eqref{c2}, we obtain 
   \begin{align} 
	T_{t}f(x)& \leq \left\| f \right\|_{\infty} \left\| \varphi \right\|_{\infty} \int_{D} \left( 1+c(1 \wedge t)^{-\frac{(d+2)}{2}}e^{-(\lambda_{2}-\lambda_{1})t} \right) e^{-\lambda_{1}t} \varphi(y) dy \nonumber\\
	&\leq \left\| f \right\|_{\infty} \left\| \varphi \right\|_{\infty}\left(e^{-\lambda_{1}t}+ce^{-\lambda_{2}t} \right)  \int_{D} \varphi(y) dy \nonumber\\
	&\leq \left\| f \right\|_{\infty} \left\| \varphi \right\|_{\infty}\left(1+c \right)e^{-\lambda_{1}t} . \nonumber
	\end{align} 
	Therefore, we get
	\begin{align}\label{c3}
	\left\| e^{\gamma t}T_{t}f \right\|_{\infty} \leq C_{0}\left\| \varphi \right\|_{\infty}^{2}\left(1+c \right)e^{-\left( \lambda_{1}-\gamma\right) t},
	\end{align} 
	and the condition \eqref{b2} in Theorem $\ref{mthm1}$  is satisfied, provided 
	 \begin{align} 
	 \mathcal{K}\int_{0}^{\infty} \left( e^{\left(q-1\right) \eta B^{H}(r)} \vee e^{(m+n-1) \eta B^{H}(r)} \right) e^{-\left( \lambda_{1}-\gamma\right)(m+n-1) r} dr<1,\nonumber
	 \end{align}  
	where $\mathcal{K}=2M(m+n+q-1)\left(C_{0}\left\| \varphi \right\|_{\infty}^{2}(1+c) \right)^{m+n-1}$ and $M = \max\left\lbrace |D|,\delta|D| \right\rbrace.$	Then the remaining proof follows by Theorem $\ref{mthm1}$.
    \end{proof}

    \section{Upper bound for $\tau$}\label{sec4}
    In this section, we obtain an upper bound $\tau^*$ for the explosion time $\tau$ under suitable assumptions  for the case $\frac{1}{2}<H<1$.  Using the fact that $\varphi(x)=0$ for $x \in \partial D,$ we have
    \begin{equation}{\label{as6}}
    \begin{aligned}
    v(s,\Delta \varphi)=\int_{D} v(s,x)\Delta\varphi(x)dx=\int_{D} \Delta v(s,x)\varphi(x)dx=\Delta v(s,\varphi)=-\lambda_{1} v(s,\varphi).
    \end{aligned}
    \end{equation}
    We recall that $\lambda_{1}$ is the first eigenvalue and $\varphi$ is the coresponding  eigenfunction  of $-\Delta$ on $D$  and $\displaystyle\int_{D}\varphi(x)dx=1.$ Let  $t\in(0,\infty)$ be some random number and $b>1$ be chosen so that 
     \begin{equation}\label{B1}
     \left.
     \begin{aligned}
     b^{q-p}e^{-\eta (m+n-1)B_{\ast}^{H}(t)} \geq \frac{ kC_{1}^{p}+\lambda_{1}C_{1}}{|D|^{1-q}},\ \ &
     b^{q-p} e^{-\eta(q-p)B_{\ast}^{H}(t)} \geq \frac{k C_{1}^{p}}{|D|^{1-q}}\\ \mbox{ and }\   
     b^{q-p} e^{-\eta(q-1)B_{\ast}^{H}(t)} &\geq \frac{2k\left(\int_{D}\varphi^{\frac{q}{q-p}}(x)dx \right)^{\frac{q-p}{p}}}{\left(\int_{D} \varphi^{p+1}(x)dx \right)^{\frac{q-p}{p}}}     
     \end{aligned}
     \right\},
     \end{equation}
    where $C_{1}=\displaystyle\sup_{x \in D} \varphi(x)$.
    \begin{theorem}\label{thm5.1}
    Suppose that $n,p,q>1,\ m\geq 0$ with $q \geq p$. For the given initial data $f(x) \geq b\varphi(x),$ $x\in D$ and for each $b>1$ such that \eqref{B1} holds.
   Then, we have the following results:
    \item [1.] If $m+n=q=\mu\ (say),$ then $\tau \leq \tau^{\ast}_{1},$ where $\tau^{\ast}_{1}$ is given by
    \begin{align}\label{ST1}
    \tau^{\ast}_{1} =& \inf \Biggl\{ t \geq 0 : \int_{0}^{t}e^{ \eta (\mu-1)B^{H}(s)+(-\lambda_{1}+\gamma)(\mu-1)s}ds\nonumber\\
    &\qquad \geq J^{1-\mu}(0)\Bigg[(\mu-1)\left(\frac{1}{2} \left(\displaystyle\int_{D}\varphi^{\frac{\mu}{\mu-p}}(x)dx \right)^{\frac{p-\mu}{p}}+\delta \left(\displaystyle\int_{D}\varphi ^{\frac{n}{n-1}}(x)dx\right)^{1-n}\right)\Bigg]^{-1}  \Biggr\}. 
    \end{align}
    \item [2.] If $m+n>q,$ let $A_{0}= \left( \frac{m+n-q}{m+n} \right) \left( \frac{m+n}{q}\right)^{\frac{q}{m+n-q}}$ and $\epsilon_{0} \leq \left(J^{q}(0)/A_{0} \right)^{\frac{q}{m+n-q}} $, then $\tau \leq \tau^{\ast}_{2},$ where $\tau^{\ast}_{2}$ is given by
    \begin{align}\label{ST2}
    &\tau^{\ast}_{2} = \inf \Bigg\{ t \geq 0 : \int_{0}^{t}\left( e^{\eta (q-1)B^{H}(s)} \wedge e^{\eta (m+n-1)B^{H}(s)} \right) e^{-(-\lambda_{1}+\gamma)(q-1)s}ds\geq 2{J^{1-q}(0)}\nonumber\\
    &\times\Bigg[((q-1)(-\lambda_{1}+\gamma)) \left( \left(\int_{D}\varphi^{\frac{q}{q-p}}(x) \right)^{\frac{p-q}{p}}+\delta \left(\epsilon_{0} - \frac{A_{0} \epsilon_{0}^{\frac{m+n}{m+n-q}}}{J^{q}(0)} \right) \left(\int_{D}\varphi ^{\frac{n}{n-1}}(x)dx\right)^{1-n}\right) \Bigg]^{-1}  \Bigg\}, 
    \end{align}
    where $J(0)=\displaystyle\int_{D}f(x)\varphi(x)dx.$
    \end{theorem}

    \begin{proof}  
    Let $\bar{v}(t,x)=b\varphi(x),$ for $t>0,$ $x\in D$, $v(0,x) \geq b \varphi(x)$,  and $v(t,x)=\bar{v}(t,x)=0$ on $\partial D.$ Firstly, we show that $v(t,x) \geq \bar{v}(t,x)=b\varphi(x),$ for each $x\in D$ and $t \geq 0.$ Let us take
    \begin{align}\label{e4}
    &I_{1}(t,x):= e^{ \eta \left(q-1\right) B^{H}(t)}\int_{D}\bar{v}^{q}(t,y)dy-ke^{ \eta (p-1)B^{H}(t)}\bar{v}^{p}(t,x) +\delta e^{\eta (m+n-1) B^{H}(t)}\bar{v}^{m}(t,x)\int_{D}\bar{v}^{n}(t,y)dy\nonumber\\
    &=e^{\eta \left(q-1\right) B^{H}(t)}b^{q}\int_{D}\varphi^{q}(y)dy-ke^{ \eta (p-1) B^{H}(t)}b^{p}\varphi^{p}(x)+\delta e^{\eta (m+n-1) B^{H}(t)}b^{m+n}\varphi^{m}(x)\int_{D}\varphi^{n}(y)dy.
    \end{align}  
    Using H\"older's inequality, we get
    \begin{align}
    \int_{D} \varphi(y) dy &\leq \left( \int_{D} \varphi^{q}(y) dy\right)^{\frac{1}{q}}\left(\int_{D}1^{\frac{q}{q-1}}dx \right)^{\frac{q-1}{q}}  \Rightarrow 
    \int_{D}\varphi^{q}(y)dy  \geq \left( \int_{D}\varphi(y)dy \right)^{q} |D|^{1-q} =|D|^{1-q}. \nonumber 
    \end{align}
    Therefore from \eqref{e4}, we have
    \begin{align}\label{e5}
    I_{1}(t,x)&\geq  e^{ \eta \left(q-1\right) B^{H}(t)}b^{q}|D|^{1-q}-ke^{ \eta (p-1) B^{H}(t)}b^{p}\varphi^{p}(x)\nonumber\\
    &\geq e^{ \eta \left(p-1\right) B^{H}(t)}b^{p} \Big[ e^{ \eta \left(q-p\right) B^{H}(t)}b^{q-p}|D|^{1-q}-k C_{1}^{p}\Big] \nonumber\\
    & \geq e^{- \eta \left(p-1\right) B^{H}_{\ast}(t)}b^{p} \Big[  e^{-\eta \left(q-p\right) B^{H}_{\ast}(t)}b^{q-p}|D|^{1-q}-kC_{1}^{p}\Big].
    \end{align}
    If $m+n \geq q$, the inequality $\eqref{e5}$  reduces to
    \begin{align}\label{com1}
    I_{1}(t,x)&\geq b \Big[ e^{- \eta\left(m+n-1\right) B^{H}_{\ast}(t)}b^{q-p}|D|^{1-q}-kC_{1}^{p}\Big].
    \end{align}
    and the conditions \eqref{B1} leads to $I_1(t,x)\geq \lambda_{1}b \varphi(x),$ for each $x \in D$ and $t \geq 0.$ Note that 
    \begin{align*}
    	\frac{\partial \bar{v}(t,x)}{\partial t}-(\Delta+\gamma)\bar{v}(t,x)=(\lambda_1-\gamma)b\varphi(x)\leq I_1(t,x)-\gamma \bar{v}(t,x)\leq I_1(t,x).
    \end{align*}
    Hence from \eqref{s1} and by comparison principle in Proposition \ref{p2}, we obtain that $v(t,x) \geq \bar{v}(t,x)=b \varphi(x),$ for each $x \in D$ and $t\geq 0.$
 
    
    Let $J(t):=\displaystyle\int_{D}v(t,x)\varphi(x)dx.$ Then, 
    \begin{align}
    J{'}(t)=\int_{D}\frac{\partial v(t,x)}{\partial t}\varphi(x)dx,\nonumber 
    \end{align}
    and by using \eqref{s1} and \eqref{as6}, we have
    \begin{align}
    J{'}(t)&=\int_{D}(\Delta+\gamma) v(t,x)\varphi(x)dx+ e^{\left(q-1\right) \eta B^{H}(t)}\int_{D}\varphi(z)dz\int_{D}v^{q}(t,y)dy\nonumber\\
    &\quad-ke^{(p-1) \eta B^{H}(t)}\int_{D}v^{p}(t,x)\varphi(x)dx +\delta e^{(m+n-1) \eta B^{H}(t)}\int_{D}v^{m}(t,z)\varphi(z)dz\int_{D}v^{n}(t,y)dy\nonumber\\ &=\left(-\lambda_{1}+\gamma\right) \int_{D} v(t,x)\varphi(x)dx+ e^{\left(q-1\right) \eta B^{H}(t)}\int_{D}v^{q}(t,x)dx-ke^{(p-1) \eta B^{H}(t)}\int_{D}v^{p}(t,x)\varphi(x)dx \nonumber\\
    &\qquad+ \delta e^{(m+n-1) \eta B^{H}(t)}\int_{D}v^{m}(t,z)\varphi(z)dz\int_{D}v^{n}(t,y)dy. \nonumber
    \end{align}
    Using H\"older's inequality, we deduce
    \begin{align}
    \int_{D}v^{p}(t,x)\varphi(x)dx &\leq \left(\int_{D}v^{q}(t,x)dx \right)^{\frac{p}{q}} \left(\int_{D}\varphi ^{\frac{q}{q-p}}(x)dx\right)^{\frac{q-p}{q}}, \nonumber\\ \Rightarrow
    \int_{D}v^{q}(t,x)dx &\geq \left( \int_{D}v^{p}(t,x)\varphi(x)dx \right)^{\frac{q}{p}} \left(\int_{D}\varphi ^{\frac{q}{q-p}}(x)dx\right)^{\frac{p-q}{p}} \nonumber
    \end{align}
    and 
    \begin{align}
    \int_{D}v^{n}(t,y)dy \geq \left( \int_{D}v(t,y)\varphi(y)dy \right)^{n} \left(\int_{D}\varphi ^{\frac{n}{n-1}}(y)dy\right)^{1-n}. \nonumber
    \end{align}
    By using Jensen's inequality, we obtain 
    \begin{align}\label{JE}
    \int_{D}v^{m}(t,z)\varphi(z)dz \geq \left( \int_{D}v(t,z)\varphi(z)dz\right)^{m}. 
    \end{align}
    Thus, it is immediate that 
    \begin{align} \label{e6}
    J{'}(t)&\geq\left(-\lambda_{1}+\gamma\right)J(t)+ e^{\left(q-1\right) \eta B^{H}(t)}\left( \int_{D}v^{p}(t,x)\varphi(x)dx \right)^{\frac{q}{p}} \left(\int_{D}\varphi ^{\frac{q}{q-p}}(x)dx\right)^{\frac{p-q}{p}}\nonumber\\
    &\quad-ke^{(p-1) \eta B^{H}(t)}\int_{D}v^{p}(t,x)\varphi(x)dx+\delta e^{(m+n-1) \eta B^{H}(t)}J^{m+n}(t)\left(\int_{D}\varphi ^{\frac{n}{n-1}}(x)dx\right)^{1-n}\nonumber\\
    &=\left(-\lambda_{1}+\gamma\right)J(t)+ e^{\left(q-1\right)\eta B^{H}(t)}\left( \int_{D}v^{p}(t,x)\varphi(x)dx \right)^{\frac{q}{p}}\Bigg[ \left(\int_{D}\varphi ^{\frac{q}{q-p}}(x)dx\right)^{\frac{p-q}{p}}\nonumber\\
    &\quad-ke^{(p-q) \eta B^{H}(t)}\left( \int_{D}v^{p}(t,x)\varphi(x)dx\right)^{\frac{p-q}{p}}\Bigg]+\delta e^{(m+n-1) \eta B^{H}(t)}J^{m+n}(t)\left(\int_{D}\varphi ^{\frac{n}{n-1}}(x)dx\right)^{1-n}.
    \end{align}
    Using assumption \eqref{B1}, we find 
    \begin{align}
    b^{q-p} e^{(q-p) \eta B^{H}(t)} \geq b^{q-p} e^{-(q-p) \eta B^{H}_{\ast}(t)}&= b^{q-p} e^{-(q-1) \eta B^{H}_{\ast}(t)} e^{(p-1) \eta B^{H}_{\ast}(t)}\nonumber\\
    &\geq b^{q-p} e^{-(q-1) \eta B^{H}_{\ast}(t)} \geq \frac{2k \left(\int_{D} \varphi^{\frac{q}{q-p}}(x) dx \right)^{\frac{q-p}{p}}}{\left( \int_{D} \varphi^{p+1}(x) dx \right)^{\frac{q-p}{p}}}. \nonumber
    \end{align}
    We know that $v(t,x)\geq b \varphi(x),$ for each $x \in D$ and $t\geq 0,$ and so
    \begin{align}
    \left(\int_{D} v^{p}(t,x)\varphi(x) dx \right)^{\frac{q-p}{p}} \geq b^{p-q}\left(\int_{D} \varphi^{p+1}(x) dx \right)^{\frac{q-p}{p}}. \nonumber 
    \end{align}
    Therefore from the above inequalities, we have
    \begin{align}
   & b^{q-p} e^{(q-p) \eta B^{H}(t)} \geq \frac{2k \left(\int_{D} \varphi^{\frac{q}{q-p}}(x) dx \right)^{\frac{q-p}{p}}}{\left( \int_{D} \varphi^{p+1}(x) dx \right)^{\frac{q-p}{p}}} \geq  \frac{2b^{q-p} k \left(\int_{D} \varphi^{\frac{q}{q-p}}(x) dx \right)^{\frac{q-p}{p}}}{\left(\int_{D} v^{p}(t,x)\varphi(x) dx \right)^{\frac{q-p}{p}}} \nonumber\\ &\Rightarrow
    ke^{(p-q) \eta B^{H}(t)}\left(\int_{D} v^{p}(t,x)\varphi(x) dx \right)^{\frac{p-q}{p}} \leq \frac{1}{2}\left(\int_{D}\varphi^{\frac{q}{q-p}}(x) dx \right)^{\frac{p-q}{p}} \nonumber\\
   &\Rightarrow \left(\int_{D}\varphi^{\frac{q}{q-p}}(x) dx \right)^{\frac{p-q}{p}}-ke^{(p-q) \eta B^{H}(t)}\left(\int_{D} v^{p}(t,x)\varphi(x) dx \right)^{\frac{p-q}{p}} \geq \frac{1}{2}\left(\int_{D}\varphi^{\frac{q}{q-p}}(x) dx \right)^{\frac{p-q}{p}}. \nonumber
    \end{align} 
    Therefore from \eqref{e6}, we deduce
    \begin{align}
    J{'}(t)& \geq \left(-\lambda_{1}+\gamma\right)J(t)+\frac{1}{2}e^{\left(q-1\right)\eta B^{H}(t)}\left( \int_{D}v^{p}(t,x)\varphi(x)dx \right)^{\frac{q}{p}}\left(\int_{D}\varphi^{\frac{q}{q-p}}(x) \right)^{\frac{p-q}{p}}\nonumber\\
    &\quad+\delta e^{(m+n-1) \eta B^{H}(t)}J^{m+n}(t)\left(\int_{D}\varphi ^{\frac{n}{n-1}}(x)dx\right)^{1-n}. \nonumber 
    \end{align}
    By using Jensen's inequality (cf. \eqref{JE}), we obtain 
    \begin{align}\label{e7}
    J{'}(t)& \geq \left(-\lambda_{1}+\gamma\right)J(t)+\frac{1}{2}e^{\left(q-1\right) \eta B^{H}(t)}J^{q}(t)\left(\int_{D}\varphi^{\frac{q}{q-p}}(x) \right)^{\frac{p-q}{p}}\nonumber\\
    &\quad+\delta e^{(m+n-1) \eta B^{H}(t)}J^{m+n}(t)\left(\int_{D}\varphi ^{\frac{n}{n-1}}(x)dx\right)^{1-n}.  
    \end{align}
    \textbf{Case 1:} If $m+n=q=\mu\ (say),$ then we have 
    \begin{align}\label{511}
    J{'}(t)& \geq \left(-\lambda_{1}+\gamma\right)J(t)+\frac{1}{2}e^{\left(\mu-1\right) \eta B^{H}(t)}J^{\mu}(t)\left(\int_{D}\varphi^{\frac{\mu}{\mu-p}}(x) \right)^{\frac{p-\mu}{p}}\nonumber\\
    &\quad+\delta e^{(\mu-1) \eta B^{H}(t)}J^{\mu}(t)\left(\int_{D}\varphi ^{\frac{n}{n-1}}(x)dx\right)^{1-n} \nonumber\\
    & = \left(-\lambda_{1}+\gamma\right)J(t)+\mathit{\widetilde{N}}e^{\left(\mu-1\right) \eta B^{H}(t)}J^{\mu}(t), 
    \end{align}
    where $\mathit{\widetilde{N}}=\frac{1}{2} \left(\displaystyle\int_{D}\varphi^{\frac{\mu}{\mu-p}}(x) \right)^{\frac{p-\mu}{p}}+\delta\left(\displaystyle\int_{D}\varphi ^{\frac{n}{n-1}}(x)dx\right)^{1-n}.$
    Here, by a comparison argument (see Theorem 1.3 of \cite{teschl} and Appendix \ref{app} below), we have $J(t) \geq I(t)$ for all $t \geq 0$, where $I(t)$ solves the differential equation
    \begin{align}
    \frac{d I(t)}{dt} = \left(-\lambda_{1}+\gamma\right)I(t)+\mathit{\widetilde{N}}e^{\left(\mu-1\right) \eta B^{H}(t)}I^{\mu}(t),\ I(0)=J(0). \nonumber
    \end{align}
    Solving the above equation, we have
    \begin{align}
    I(t)=e^{(-\lambda_{1}+\gamma) t}\left\lbrace I^{1-\mu}(0)-\mathit{\widetilde{N}}(\mu-1)\int_{0}^{t} e^{ \eta (\mu-1)B^{H}(t)+(-\lambda_{1}+\gamma)(\mu-1)s}ds \right\rbrace^{-\frac{1}{\mu-1}}. \nonumber 
    \end{align}
    For the above equation, the blow-up time is given by \eqref{ST1}. 
 \vskip 0.2cm 
    \noindent\textbf{Case 2:} If $m+n>q$, then we infer
    \begin{align}\label{b3}
    J'(t)& \geq \left(-\lambda_{1}+\gamma\right)J(t)+\left( e^{\left(q-1\right) \eta B^{H}(t)} \wedge e^{(m+n-1)\eta B^{H}(t)}\right)\Bigg[ \frac{1}{2} J^{q}(t)\left(\int_{D}\varphi^{\frac{q}{q-p}}(x) \right)^{\frac{p-q}{p}}\nonumber\\
    &\quad+\delta J^{m+n}(t)\left(\int_{D}\varphi ^{\frac{n}{n-1}}(x)dx\right)^{1-n} \Bigg].
    \end{align}
    The Young inequality states that if $1<b_{1}<\infty$, $\delta>0$ and $a_{0}=\frac{b_{1}}{b_{1}-1}$, then 
    \begin{eqnarray} \label{na10}
    xy\leq \frac{\delta^{a_{0}}x^{a_{0}}}{a_{0}}+\frac{\delta^{-b_{1}}y^{b_{1}}}{b_{1}}, \ x,y\geq 0. 
    \end{eqnarray}
    By setting $b_{1}=\frac{m+n}{q}, \ y=J^{q}(t), \ x=\epsilon, \ \delta=\left( \frac{m+n}{q} \right)^{\frac{q}{m+n}}$ and using the fact that $q<m+n$ in (\ref{na10}), it follows that for any $\epsilon>0,$
    \begin{align}
    \epsilon J^{q}(t) &\leq \frac{\left(\left(\frac{m+n}{q} \right)^{\frac{q}{m+n}}  \right)^{\frac{m+n}{m+n-q}}\epsilon^{\frac{m+n}{m+n-q}}}{\frac{m+n}{m+n-q}} + \frac{\left( \left(\frac{m+n}{q} \right)^{\frac{q}{m+n}}  \right)^{-\frac{m+n}{q}}\left(J^{q}(t) \right)^{\frac{m+n}{q}}  }{\frac{m+n}{q}}  \nonumber \\
    \Rightarrow J^{m+n}(t) &\geq \epsilon J^{q}(t) - A_{0} \epsilon^{\frac{m+n}{m+n-q}}.\nonumber 
    \end{align}
    Therefore from \eqref{b3}, we have
    \begin{align}
    J'(t)& \geq \left(-\lambda_{1}+\gamma\right)J(t)+\left( e^{\left(q-1\right) \eta B^{H}(t)} \wedge e^{(m+n-1)\eta B^{H}(t)}\right)\Bigg[ \frac{1}{2} J^{q}(t)\left(\int_{D}\varphi^{\frac{q}{q-p}}(x) \right)^{\frac{p-q}{p}}\nonumber\\
    &\quad+\delta \left(\epsilon_{0} J^{q}(t) - A_{0} \epsilon_{0}^{\frac{m+n}{m+n-q}} \right) \left(\int_{D}\varphi ^{\frac{n}{n-1}}(x)dx\right)^{1-n} \Bigg]\nonumber\\
    &\geq \left(-\lambda_{1}+\gamma\right)J(t)+\left( e^{\left(q-1\right) \eta B^{H}(t)} \wedge e^{(m+n-1)\eta B^{H}(t)}\right)J^{q}(t)\Bigg[ \frac{1}{2} \left(\int_{D}\varphi^{\frac{q}{q-p}}(x) \right)^{\frac{p-q}{p}}\nonumber\\
    &\quad+\delta \left(\epsilon_{0} - \frac{A_{0} \epsilon_{0}^{\frac{m+n}{m+n-q}}}{J^{q}(0)} \right) \left(\int_{D}\varphi ^{\frac{n}{n-1}}(x)dx\right)^{1-n} \Bigg].\nonumber  
    \end{align}
    Here, by a comparison argument (see Theorem 1.3 of \cite{teschl} and Appendix), we have $J(t) \geq I(t)$ for all $t \geq 0$, where $I(t)$ solves the differential equation
    \begin{align}
    \frac{d I(t)}{dt} &= \left(-\lambda_{1}+\gamma\right)J(t)+\left( e^{\left(q-1\right) \eta B^{H}(t)} \wedge e^{(m+n-1)\eta B^{H}(t)}\right)J^{q}(t)\Bigg[ \frac{1}{2} \left(\int_{D}\varphi^{\frac{q}{q-p}}(x) \right)^{\frac{p-q}{p}}\nonumber\\
    &\quad+\delta \left(\epsilon_{0} - \frac{A_{0} \epsilon_{0}^{\frac{m+n}{m+n-q}}}{J^{q}(0)} \right) \left(\int_{D}\varphi ^{\frac{n}{n-1}}(x)dx\right)^{1-n} \Bigg],\
    I(0)=J(0). \nonumber   
    \end{align}
    Solving the above equation, we have
    \begin{align}
    I(t)&=e^{(q-1)(-\lambda_{1}+\gamma)t}\Bigg\{ I^{1-q}(0)-(q-1)(-\lambda_{1}+\gamma)D\nonumber\\
    &\qquad\times \int_{0}^{t} \left( e^{\left(q-1\right) \eta B^{H}(s)} \wedge e^{(m+n-1)\eta B^{H}(s)}\right)e^{-(q-1)(-\lambda_{1}+\gamma)s} ds  \Bigg\}^{-\frac{1}{p-1}},
    \end{align}
    where $ D=\displaystyle\frac{1}{2} \left(\int_{D}\varphi^{\frac{q}{q-p}}(x) \right)^{\frac{p-q}{p}}+\delta \left(\epsilon_{0} - \frac{A_{0} \epsilon_{0}^{\frac{m+n}{m+n-q}}}{J^{q}(0)} \right) \left(\int_{D}\varphi ^{\frac{n}{n-1}}(x)dx\right)^{1-n}.$
    For the above inequality, the blow-up time is given by \eqref{ST2},
    which completes the proof of the theorem.
    \end{proof}
    
    \subsection{Probability of blow-up when $\frac{1}{2}<H<1$}\label{prb1}
    In this section, we estimate the blow-up probability of the solution to \eqref{b1} for  $\frac{1}{2}<H<1$. 
    
    First let us recall some basic concepts of Malliavin calculus  (for more details see \cite{dung, nualart}, etc.). 
        Let $W=\{W(t)\}_{t\geq 0}$ be defined on a complete probability space $\left( \Omega, \mathcal{F}, (\mathcal{F}_{t})_{t \geq 0} , \mathbb{P} \right).$ For $h \in \mathrm{L}^{2}(\mathbb{R^{+}}),$  the Wiener integral $\mathcal W(h)$ is denoted by $$\mathcal W(h)=\int_{0}^{\infty} h(s) dW(s).$$ 
    Let $\mathcal{S}$ denote the dense subset of $\mathrm{L}^{2}(\Omega, \mathcal{F},\mathbb{P})$ consisting of smooth random variables of the form 
    \begin{align}\label{p1}
    F=f(\mathcal W(h_{1}),\mathcal W(h_{2}),\ldots, \mathcal W(h_{n})), 
    \end{align}
    where $n\in \mathbb{N},\ f \in \mathrm{C}_{0}^{\infty}(\mathbb R^n)$ and $h_{1},h_{2},\ldots, h_n \in \mathrm{L}^{2}(\mathbb{R}^{+}).$
    If $F$ is of the form \eqref{p1},  we define its Malliavin derivative as the process  $ DF:=\{D_{t}F\}_{t \geq 0}$  given by 
    \begin{align*}
    D_{t}F = \sum_{k=1}^{n} \frac{\partial f}{\partial x_{k}} \left( \mathcal W(h_{1}),\mathcal W(h_{2}),..., \mathcal W(h_{n}) \right) h_{k}(t). 
    \end{align*}
    For any $1\leq p< \infty,$ we shall denote by $\mathbb{D}^{1,p}$, the closure of $\mathcal{S}$ with respect to the norm 
    \begin{align*}
    \|F\|_{1,p}^{p} := \mathbb{E}\left[| F |^{p}\right]+ \mathbb{E}\left[ \int_{0}^{\infty}|\mathrm{D}_{t}F|^{p}dt\right]. 
    \end{align*}   
    A random variable $F$ is said to be \emph{Malliavin differentiable} if it belongs to $\mathbb{D}^{1,2}.$ 
    
    We establish an upper bound for probabilities of the form
    \begin{align}
    \mathbb{P}\left[ \int_{0}^{\infty} \exp\{ -as+\sigma X_{s} \} ds <x \right], \nonumber 
    \end{align}
    where $x>0,\ \sigma>0$ and $a$ are real numbers, $\{X_s\}_{s\geq 0}$ is a continuous stochastic processes, which also contains fBm $B^{H}$.
    \begin{assumption} \label{as2}
    The stochastic process $\{X_{t}\}_{t \geq 0}$ is $\mathcal{F}_t$-adapted and satisfies the following properties:   
    \end{assumption}
    \begin{itemize}	
    \item [1.] $\displaystyle\int_{0}^{\infty} \exp\{-as\}\mathbb{E}\left[ \exp\{\sigma X_{s}\}\right] ds < +\infty.$
    \item [2.] For each $t \geq 0, \ X_{t} \in \mathbb{D}^{1,2}.$
    \item [3.] There exists a function $f : \mathbb{R}^{+} \rightarrow \mathbb{R}^{+}$ such that $\displaystyle\lim_{t \rightarrow \infty} f(t)= +\infty$ and for each $x>0,$
    \begin{align*}
    \sup_{t\geq 0} \frac{\displaystyle\sup_{s\in [0,t]} \displaystyle\int_{0}^{s} |D_{\theta}X_{s}|^{2} d\theta}{\left( \ln \left( x+1\right)+f(t) \right)^{2}} \leq M_{x}< +\infty \ \ \mathbb{P}\text{-a.s.}
    \end{align*}
    \end{itemize}

    \begin{theorem}\cite[Theorem 3.1]{dung}\label{mt6} 
    Suppose that Assumption  \ref{as2} holds. We have 
    \begin{align}\label{m7}
    \mathbb{P} \left[ \int_{0}^{\infty} \exp\{-as+\sigma X_{s} \} ds < x \right] \leq \exp \left\lbrace -\frac{(m_{x} -1)^{2}}{2 \sigma^{2}M_{x}}\right\rbrace,
    \end{align}
    where 
    \begin{align*}
    m_{x}=\mathbb{E}\left[\sup_{t \geq 0} \frac{\ln \left(  \int_{0}^{t} \exp\{-as+\sigma X_{s} \} ds+1 \right) +f(t)}{\ln \left( x+1\right)+f(t)}\right]\geq 1.
    \end{align*}
    \end{theorem}
     The following theorem is useful to obtain a lower bound for the probability of blow-up solution of the equation \eqref{b1} with $\frac{1}{2}<H<1.$

     \begin{theorem}
     Suppose that $\alpha > H$ with $q > p>1, m+n = q$. For each positive initial value $f(x) \geq b\varphi(x),$ $x\in D$ and for each $b>1$ such that \eqref{B1} holds,  and the blow-up times $\tau^{\ast}_{1}$ and $\tau^{\ast}_{2}$ of \eqref{b1} as defined in Theorem \ref{thm5.1} with $J(0)=\int_{D}f(x) \varphi(x)dx,$ we have the following results:
     \begin{itemize}
     \item[1.] If $\lambda_{1}<\gamma,$ then $\mathbb{P}(\tau^{\ast}_{1}= \infty)=0$ for any non-trivial positive solution $u(\cdot,\cdot)$ of the equation \eqref{b1}, that is the solution  $u(\cdot,\cdot)$  blows-up in finite-time $\mathbb{P}\text{-a.s. }$
     \item[2.] 
     If $\lambda_{1} > \gamma$ and $m+n=q=\mu\ (say)$, then a lower bound for the probability of blow-up solution of \eqref{b1} is given by 
    
   \begin{align}
   &\mathbb{P} (\tau < \infty) \nonumber\\&\geq 1-\exp \Bigg\{ -\frac{1}{2 \rho^{2}_{1}} \left(\log \left( \frac{1}{2} \left(\int_{D}\varphi^{\frac{\mu}{\mu-p}}(x)dx \right)^{\frac{p-\mu}{p}}+\delta\left(\int_{D}\varphi ^{\frac{n}{n-1}}(x)dx\right)^{1-n}+1\right)  \right)^{\frac{2H}{\alpha}-2} \nonumber\\
   &\qquad\qquad\times \left( \frac{\alpha-H}{\alpha} \right)^{2-\frac{2H}{\alpha}} \left(N(H)-1 \right)^{2}  \Bigg\}, \nonumber 
   \end{align}
    where    $ \rho_{1}=\eta (\mu-1)$ and 
   \begin{align}\label{516}
   N(H):= \mathbb{E}\left[\sup_{t \geq 0} \frac{\ln \left( \displaystyle 1+\int_{0}^{t} \exp\{(-\lambda_{1}+\gamma)(\mu-1)s-\frac{\rho^{2}_{1}}{2}s^{2H}+\rho_{1} B^{H}(s) \} ds \right) +t^{\alpha}}{\ln \left(\frac{1}{2} \left(\int_{D}\varphi^{\frac{\mu}{\mu-p}}(x)dx \right)^{\frac{p-\mu}{p}}+\delta\left(\int_{D}\varphi ^{\frac{n}{n-1}}(x)dx\right)^{1-n}+1\right)+t^{\alpha}}\right].
   \end{align}
     \end{itemize}
     \end{theorem}
     \begin{proof}
     \textbf{Case 1:} Assume that $\lambda_{1}<\gamma,$ and recall the law of  iterated logarithm for $B^{H}$ which is given by (see more details in \cite{law})
     \begin{align}
     \liminf_{t \rightarrow \infty} \frac{B^{H}(t)}{t^{H}\sqrt{2 \log \log t}}=-1 \ \ \mbox{and}\ \ \limsup_{t \rightarrow \infty} \frac{B^{H}(t)}{t^{H}\sqrt{2 \log \log t}}=1, \ \ \mathbb{P}\text{- a.s.} \nonumber 
     \end{align} 
     If $m+n=q,$ then by Lemma \ref{l2} and \eqref{ST1}, we have $\mathbb{P}(\tau^{\ast}_{1}= \infty)=0$, 
      and any non-trivial positive solution of the equation \eqref{b1} blows-up in finite-time $\mathbb{P}\text{-a.s. }$\\
          	
  \noindent   \textbf{Case 2 :} If $m+n=q=\mu\ (say)$ then by using the definition of $\tau^{\ast}_{1}$ in \eqref{ST1}, we have      
     \begin{align}
     &\mathbb{P}\left( \tau^{\ast}_{1}=\infty \right)\nonumber\\
     &=\mathbb{P} \Bigg( \int_{0}^{\infty} e^{\rho_{1} B^{H}(s)+(-\lambda_{1}+\gamma)(\mu-1)s}ds < \frac{J^{1-\mu}(0)}{(\mu-1)\left(\frac{1}{2} \left(\int_{D}\varphi^{\frac{q}{q-p}}(x)dx \right)^{\frac{p-q}{p}}+\delta \left(\int_{D}\varphi ^{\frac{n}{n-1}}(x)dx\right)^{1-n}\right)}  \Bigg) \nonumber\\
     & \leq \mathbb{P} \Bigg( \int_{0}^{\infty} e^{\rho_{1} B^{H}(s)-\frac{\rho_{1}^{2}}{2}s^{2H}+(-\lambda_{1}+\gamma)(\mu-1)s}ds < \frac{J^{1-\mu}(0)}{(\mu-1)\left(\frac{1}{2} \left(\int_{D}\varphi^{\frac{q}{q-p}}(x)dx \right)^{\frac{p-q}{p}}+\delta \left(\int_{D}\varphi ^{\frac{n}{n-1}}(x)dx\right)^{1-n}\right)}  \Bigg). \nonumber 
     \end{align} 
     We choose $X_{t}=-\frac{\rho_{1}}{2}t^{2H}+B^{H}(t),\ \rho_{1}=\eta(\mu-1).$ First we show that the stochastic process $\{X_{t}\}_{t\geq 0}$ satisfies the conditions of Theorem \ref{mt6}. Consider
     \begin{align}
     \int_{0}^{\infty} \mathbb{E}\left(  e^{(-\lambda_{1}+\gamma)(\mu-1)s-\frac{\rho^{2}_{1}}{2}s^{2H}+\rho_{1} B^{H}(s)}\right) ds&=\int_{0}^{\infty}  e^{(-\lambda_{1}+\gamma)(\beta-1)s-\frac{\rho^{2}_{1}}{2}s^{2H}}\mathbb{E}\left( e^{\rho_{1}B^{H}(s)}\right) ds \nonumber\\
     &=\int_{0}^{\infty}  e^{(-\lambda_{1}+\gamma)(\mu-1)s} ds=\frac{1}{(-\lambda_{1}+\gamma)(\mu-1)}< \infty,
     \end{align}
     and 
     \begin{align}
     D_{\theta} X_{t}= D_{\theta}\left(-\frac{\rho_{1}}{2}t^{2H}+B^{H}(t) \right) = D_{\theta} B^{H}(t)= K_{H}(t,\theta), \ \theta \leq t.
     \end{align}
     Moreover, we have
     \begin{align}\label{520}
     \sup_{s \in [0,t]}\int_{0}^{s} \left|D_{\theta} X_{s} \right|^{2} d \theta&=\sup_{s \in [0,t]} \int_{0}^{s} K^{2}_{H}(s,\theta) d \theta = \sup_{s \in [0,t]} \mathbb{E}\left|B^{H}(s) \right|^{2} =t^{2H}.  
     \end{align}
     Now we take $f(t)=t^{\alpha}$ with $\alpha>H.$ Therefore the stochastic process $\{X_{t}\}_{t\geq 0}$ satisfies the conditions of Theorem \ref{mt6}. Moreover,
     \begin{align}
     M_{H}:= \sup_{t \geq 0} t^{2H} \left\{\log \left(\frac{J^{1-\mu}(0)}{(\mu-1)\left(\frac{1}{2} \left(\int_{D}\varphi^{\frac{\mu}{\mu-p}}(x)dx \right)^{\frac{p-\mu}{p}}+\delta\left(\int_{D}\varphi ^{\frac{n}{n-1}}(x)dx\right)^{1-n}\right)}+1\right)  +t^{\alpha} \right\}^{-2}, \nonumber
     \end{align}
     and by a simple calculation, we get 
     \begin{align}
     M_{H}=\left( \frac{\alpha-H}{\alpha} \right)^{2-\frac{2H}{\alpha}}\left(\log \left( \frac{1}{2} \left(\int_{D}\varphi^{\frac{\mu}{\mu-p}}(x)dx \right)^{\frac{p-\mu}{p}}+\delta \left(\int_{D}\varphi ^{\frac{n}{n-1}}(x)dx\right)^{1-n}+1\right)  \right)^{\frac{2H}{\alpha}-2} . \nonumber  
     \end{align}
     Then by Theorem \ref{mt6}, we have 
     \begin{align}
     &\mathbb{P} \Bigg(\int_{0}^{\infty} e^{\rho_{1} B^{H}(s)-\frac{\rho^{2}_{1}}{2}s^{2H}+(-\lambda_{1}+\gamma)(\mu-1)s}ds < \frac{J^{1-\mu}(0)}{(\mu-1)\left(\frac{1}{2} \left(\int_{D}\varphi^{\frac{\mu}{\mu-p}}(x)dx \right)^{\frac{p-\mu}{p}}+\delta \left(\int_{D}\varphi ^{\frac{n}{n-1}}(x)dx\right)^{1-n}\right)} \Bigg) \nonumber\\
     &\leq \exp \Bigg\{ -\frac{1}{2 \rho^{2}_{1}} \left(\log \left( \frac{1}{2} \left(\int_{D}\varphi^{\frac{\mu}{\mu-p}}(x)dx \right)^{\frac{p-\mu}{p}}+\delta\left(\int_{D}\varphi ^{\frac{n}{n-1}}(x)dx\right)^{1-n}+1\right)  \right)^{\frac{2H}{\alpha}-2}\nonumber\\
     &\qquad\qquad\times \left( \frac{\alpha-H}{\alpha} \right)^{2-\frac{2H}{\alpha}} \left(N(H)-1 \right)^{2}  \Bigg\}. \nonumber 
     \end{align}
     Therefore, we obtain
     \begin{align}
     &\mathbb{P} (\tau < \infty)=1-\mathbb{P}(\tau = \infty)\nonumber\\
     & \geq 1-\exp \Bigg\{ -\frac{1}{2 \rho^{2}_{1}} \left(\log \left( \frac{1}{2} \left(\int_{D}\varphi^{\frac{\mu}{\mu-p}}(x)dx \right)^{\frac{p-\mu}{p}}+\delta\left(\int_{D}\varphi ^{\frac{n}{n-1}}(x)dx\right)^{1-n}+1\right)  \right)^{\frac{2H}{\alpha}-2} \nonumber\\
     &\qquad\qquad\times\left( \frac{\alpha-H}{\alpha} \right)^{2-\frac{2H}{\alpha}}  \left(N(H)-1 \right)^{2}  \Bigg\}, \nonumber 
     \end{align}
     where $\rho_{1}=\eta(\mu-1)$ and $N(H)$ is defined in \eqref{516}, which completes the proof of theorem. 
    \end{proof}
     \begin{Rem} For the case $m+n>q$, we have the blow-up time given by \eqref{ST2}
     and  the probability distribution for estimating the probability of blow-up solution for the case of fractional Brownian motion is not known and it will be a part of our future work.
     \end{Rem}

     \section{For the case $H=\frac{1}{2}$}
     In this section, we consider the case $H=\frac{1}{2}$ and  $B^{\frac{1}{2}}(\cdot)=W(\cdot),$ where $W(\cdot)$ is a standard one-dimensional Brownian motion. We are devoted to establish the lower and upper bounds for the      probability of the blow-up solution to the equation \eqref{b1} by using the methods developed in \cite{dung} and \cite{Eug2017}. By using the random transformation $$v(t,x) = \exp\{-\eta W(t)\}u(t,x),$$ for $t \geq 0,\ x \in D$,  the equation \eqref{b1} is transformed into the following random PDE:
     \begin{equation}\label{ss1}
     \left\{
     \begin{aligned}
     \frac{\partial v(t,x)}{\partial t}&=\left( \Delta+\gamma-\frac{\eta^{2}}{2}\right) v(t,x)+ e^{\left( q-1\right)\eta W(t) }\int_{D}v^{q}(t,x)dx-ke^{\left( p-1\right)\eta W(t) }v^{p}(t,x)\\
     &\qquad +\delta e^{(m+n-1) \eta W(t)}v^{m}(t,x)\int_{D} v^{n}(t,x)dx,\\
     v(t,x)&=0, \ \ x \in D, \\
     v(0,x)&=f(x), \ \  x \in D.  
     \end{aligned}
     \right.
     \end{equation}
     Taking $\Lambda=\frac{ \eta^{2}}{2}-\gamma,$ the results obtained in previous sections are valid by replacing the constant $\gamma$ by $-\Lambda.$ Next, we recall that $X(\alpha,\beta_{1})$ is said to be a \emph{Gamma random variable} with parameters $\alpha>0,\beta_{1}>0$ if its density function $\widetilde{f}_{\alpha,\beta_{1}}(x)$ is given by (cf. \cite{li})
     \begin{equation}
     \widetilde{f}(x) = \left\{
     \begin{aligned}
     &\frac{x^{\alpha-1}}{\beta_{1}^{\alpha}\varGamma (\alpha)} \exp\left\lbrace-\frac{x}{\beta_{1}} \right\rbrace , \ x\geq 0 , \\
     &\hspace{.3 in} 0, \hspace{0.9 in} \ x<0. 
     \end{aligned}
     \right.
     \end{equation}
     The following lemma is very useful to estimate the probability of blow-up of positive solutions $v(\cdot,\cdot)$ of the equation \eqref{ss1}.
    \begin{lemma} (\cite{yor2005}, \cite{revuz1999}, \cite[Chapter 6, Corollary 1.2]{yor2001}){\label{ll1}}
    For any $\alpha>0,$ the exponential functional $\displaystyle\int_{0}^{\infty} e^{2(W(t)-\alpha t)} dt$ 	is distributed as $(2X(\alpha,1))^{-1},$ where $W(\cdot)$ is a one-dimensional standard Brownian motion. 
    \end{lemma}
    The following theorem gives  the lower bounds for the finite-time blow-up of the solution $u(\cdot, \cdot)$ of the equation \eqref{ss1}.
	\begin{theorem} \label{thm6.1}
	Assume that $p,q,n>1,$ $m \geq 0$ with $m+n \geq q$. If $f$ is a non-negative bounded function and $|D|>k$, then $\sigma_{\ast}\leq\tau$, where $\sigma_{\ast}$ is given by
	\begin{align}
	\sigma_{\ast} = \inf \Bigg\{ t\geq 0 :\int_{0}^{t}\left( e^{\eta \left( q-1\right)  W(r)} \vee e^{\eta (m+n-1) W(r)}\right)  e^{-\Lambda(m+n-1) r}dr \geq  \frac{1}{2M(m+n-1) \|f\|_{\infty}^{m+n-1}} \Bigg\}, \nonumber
	\end{align}
		where $M = \max\left\lbrace |D|,\delta |D| \right\rbrace.$
	\end{theorem}
    The proof of Theorem $\ref{thm6.1}$ follows from Theorem \ref{t2}.
    \noindent \\
	Let us denote 
	\begin{align}
	W_{\ast}(t) = \sup_{0 \leq s \leq t} |W(s)|,\ \mbox{for each}\ t> 0. \nonumber 
	\end{align}
	From  \cite[p.96]{Karatzas}, we infer that for any $A>0$ and $t>0$
	\begin{align}
	\mathbb{P}(W_{\ast}(t) \geq A) \leq \frac{4 \sqrt{t}}{A \sqrt{2 \pi}}e^{-\frac{A^{2}}{2t}}. \nonumber 
	\end{align}
	 \noindent  Hence $$ \mathbb{P}(W_{\ast}(t)< \infty)=1-\lim_{A \rightarrow \infty} \mathbb{P}(W_{\ast}(t)> A)=1,\ t>0.$$
	 Let us denote $U_{t}=\{\omega \in \Omega : W_{\ast}(t)= \infty\},$ for every $t>0.$ Then, we have $\mathbb{P}(U_{t})=0,$ for every $t>0.$ If we take $t=1,2,\ldots,$ it is clear that $U_{t} \subset U_{m},$ for $t \geq m.$ Define $U=\displaystyle\lim_{m \rightarrow \infty} U_{m},$ we have $P(U)=\displaystyle\lim_{m\to \infty}\mathbb{P}(U_{m})=0.$ Therefore, for all $\omega \in U, $ $W_{\ast}(t;\omega)<\infty,$ for all $t>0.$ From the above observation, without loss of generality, we can assume that for all $\omega \in \Omega,$
	 \begin{align} 
	 W_{\ast}(t;\omega)<\infty,\ \mbox{for all}\ t>0. \nonumber 
	 \end{align}
	Let $N>0$ be a constant and define the stopping time as follows:
	$$\tau_{N}(\omega)=\inf \left\lbrace t>0:\ |W(t;\omega)|\geq N\right\rbrace.$$
	Clearly, $$\left\lbrace \omega \in \Omega :\ \tau_{N}(\omega) \leq t \right\rbrace=\left\lbrace \omega \in \Omega:\ W_{\ast}(t) \geq N \right\rbrace.$$
	Let  $t\in(0,\infty)$ be some random number and $b>1$ be chosen so that 
	\begin{equation}\label{Bs1}
	\left.
	\begin{aligned}
	b^{q-p}e^{-\eta (m+n-1)W_{\ast}(t)} \geq \frac{ kC_{1}^{p}+(\lambda_{1}+\Lambda) C_{1}}{|D|^{1-q}},\ \ &
	b^{q-p} e^{-\eta(q-p)W_{\ast}(t)} \geq \frac{k C_{1}^{p}}{|D|^{1-q}}\\ \mbox{ and }\   
	b^{q-p} e^{-\eta(q-1)W_{\ast}(t)} &\geq \frac{2k\left(\int_{D}\varphi^{\frac{q}{q-p}}(x)dx \right)^{\frac{q-p}{p}}}{\left(\int_{D} \varphi^{p+1}(x)dx \right)^{\frac{q-p}{p}}}     
	\end{aligned}
	\right\},
	\end{equation} 
	where $C_{1}=\displaystyle\sup_{x \in D} \varphi(x)$.	From \eqref{com1}, we have
	 \begin{align}
	 I_{1}(t,x)&\geq b \Big[ e^{-\eta\left(m+n-1\right) W_{\ast}(t)}b^{q-p}|D|^{1-q}- kC_{1}^{p}\Big] \nonumber
	 \end{align}
	  and the conditions \eqref{Bs1} leads to $I_1(t,x)\geq b(\lambda_{1}+\Lambda) \varphi(x),$ for each $x\in D$ and $t \geq 0.$ Note that 
	  \begin{align*}
	  \frac{\partial \bar{v}(t,x)}{\partial t}-\Delta\bar{v}(t,x)=\lambda_1 b\varphi(x)\leq I_1(t,x)-\Lambda \bar{v}(t,x).
	  \end{align*}
	  Hence from \eqref{ss1} and by comparison principle in Proposition \ref{p2}, we obtain $v(t,x) \geq \bar{v}(t,x)=b \varphi(x),$ for each $x\in D$ and $t \geq 0.$
	  
	 The following theorem provides  upper bounds for the finite-time blow-up. Further, we estimate a lower bound for the probability of blow-up  solution of the random PDE \eqref{b1} when $H=\frac{1}{2}$ and $\Lambda=\frac{ \eta^{2}}{2}-\gamma.$ 
	\begin{theorem}\label{thm6.3}
	Suppose that $n,p,q>1,\ m\geq 0$ with $q\geq p$ and for given initial value $f(x) \geq b \varphi(x),\ x \in D$ and choose $b>1$ such that \eqref{Bs1} holds. Then, we have the following results:
	\item [1.] If $m+n=q=\mu\ (say).$ Then $\tau \leq \sigma^{\ast}_{1},$ where $\sigma^{\ast}_{1}$ is given by
	\begin{align}
	\sigma^{\ast}_{1} =& \inf \Biggl\{ t \geq 0 : \int_{0}^{t}e^{ \rho_{1}W(s)-(\lambda_{1}+\Lambda)(\mu-1)s}ds\nonumber\\
	&\qquad\quad \geq J^{1-\mu}(0)\left[(\mu-1)\left(\frac{1}{2} \left(\displaystyle\int_{D}\varphi^{\frac{\mu}{\mu-p}}(x)dx \right)^{\frac{p-\mu}{p}}+\delta \left(\displaystyle\int_{D}\varphi ^{\frac{n}{n-1}}(x)dx\right)^{1-n}\right)\right]^{-1}  \Biggr\}. \nonumber 
	\end{align}
	 Moreover, there exists $\theta_{1}>0$ such that a lower bound for the probability of blow-up solution of the equation \eqref{ss1} is given by
	 \begin{align}
	 &\mathbb{P}\left( \tau<\infty \right) \geq \mathbb{P}\left( \sigma^{\ast}_{1}<\infty \right)\nonumber\\
	 &=1-\mathbb{P} \Bigg( 2(\mu-1)\left(\frac{1}{2} \left(\int_{D}\varphi^{\frac{\mu}{\mu-p}}(x)dx \right)^{\frac{p-\mu}{p}}+\delta\left(\int_{D}\varphi ^{\frac{n}{n-1}}(x)dx\right)^{1-n}\right)<X(\theta_{1},1)\rho^{2}_{1}J^{1-\mu}(0)\Bigg), \nonumber
	 \end{align}
	 where $\theta_{1}=\displaystyle\frac{2(\lambda_{1}+\Lambda)(\mu-1)}{\rho^{2}_{1}}$ and $\rho_{1}= \eta (\mu-1).$
	\item [2.] If $m+n>q>1$, let $A_{0}= \left( \frac{m+n-q}{m+n} \right) \left( \frac{m+n}{q}\right)^{\frac{q}{m+n-q}}$ and $\epsilon_{0} \leq \left(J^{q}(0)/A_{0} \right)^{\frac{q}{m+n-q}} $. Then $\tau \leq \sigma^{\ast \ast}_{2},$ where $\sigma^{\ast \ast}_{2}$ is given by
	\begin{align}
   \sigma^{\ast \ast}_{2} = \inf \Biggl\{ t \geq 0 : \int_{0}^{t}e^{-(\eta(q-1) W(s)-(\lambda_{1}+\Lambda)(q-1)s)}  \mathbf{1} _{\{W(s) \geq 0\}}ds \geq a_{1} \Biggr\}. \nonumber 
   \end{align}	
   Moreover, a lower bound for the probability of blow-up solution of the equation \eqref{ss1} is given by
   \begin{align}
	&\mathbb{P}\left( \tau<\infty \right) \geq \mathbb{P}\left( \sigma^{\ast \ast}_{2}<\infty \right)=\frac{8 (\lambda_{1}+\Lambda)(q-1)}{(\eta(q-1))^{2}} \sum_{n \geq 1} \frac{\exp \Bigg\{- \frac{(\eta(q-1))^{2} a_{1}}{8}j^{2}_{\frac{2 (\lambda_{1}+\Lambda)(q-1)}{(\eta(q-1))^{2}}-1, n} \Bigg\}}{j^{2}_{\frac{2 (\lambda_{1}+\Lambda)(q-1)}{(\eta(q-1))^{2}}-1, n}}, \nonumber
	\end{align}
	where \\
	$a_{1}=2{J^{1-q}(0)}\Bigg[(q-1)(-\lambda_{1}+\gamma) \left( \left(\int_{D}\varphi^{\frac{q}{q-p}}(x) \right)^{\frac{p-q}{p}}+\delta \left(\epsilon_{0} - \frac{A_{0} \epsilon_{0}^{\frac{m+n}{m+n-q}}}{J^{q}(0)} \right) \left(\int_{D}\varphi ^{\frac{n}{n-1}}(x)dx\right)^{1-n}\right) \Bigg]^{-1}$, \\ $\Big\{j_{\frac{2 (\lambda_{1}+\Lambda)(q-1)}{(\eta(q-1))^{2}}-1, n}\Big\}_{n \geq 1}$ is an increasing sequence of all positive zeros of the Bessel function of the first kind of order $\frac{2 (\lambda_{1}+\Lambda)(q-1)}{(\eta(q-1))^{2}}-1>-1$ (see more details in \cite{Eug2017}) and $J(0)=\displaystyle\int_{D}f(x) \varphi(x)dx.$
	\end{theorem}

	\begin{proof}
	 \textbf{Case 1:} If $m+n=q=\mu\ \mbox{(say)}$ by proceeding in the same way as in the case 1 of Theorem \ref{thm5.1} and by using \eqref{ss1}, we obtain
	 \begin{align}
	 \sigma^{\ast}_{1} =& \inf \Biggl\{ t \geq 0 : \int_{0}^{t}e^{ \rho_{1}W(s)-(\lambda_{1}+\Lambda)(\mu-1)s}ds\nonumber\\
	 &\qquad\quad \geq J^{1-\mu}(0)\left[(\mu-1)\left(\frac{1}{2} \left(\displaystyle\int_{D}\varphi^{\frac{\mu}{\mu-p}}(x)dx \right)^{\frac{p-\mu}{p}}+\delta \left(\displaystyle\int_{D}\varphi ^{\frac{n}{n-1}}(x)dx\right)^{1-n}\right)\right]^{-1}  \Biggr\}, \nonumber 
	 \end{align}
	where $\rho_{1}=\eta (\mu-1).$ From the definition of $\sigma^{\ast}_{1}$, we infer
	\begin{align}
	&\mathbb{P}\left( \sigma^{\ast}_{1}=\infty \right)\nonumber\\
	&= \mathbb{P} \Bigg( \int_{0}^{\infty}e^{\rho_{1}W(s)-(\lambda_{1}+\Lambda)(\mu-1)s}ds <\frac{J^{1-\mu}(0)}{(\mu-1)\left(\frac{1}{2} \left(\displaystyle\int_{D}\varphi^{\frac{\mu}{\mu-p}}(x)dx \right)^{\frac{p-\mu}{p}}+\delta \left(\displaystyle\int_{D}\varphi ^{\frac{n}{n-1}}(x)dx\right)^{1-n}\right)} \Bigg) \nonumber\\
	&=\mathbb{P} \Bigg( \int_{0}^{\infty}e^{2W\left( \frac{\rho^{2}_{1}}{4}s\right)-(\lambda_{1}+\Lambda)(\mu-1)s}ds < \frac{J^{1-\mu}(0)}{(\mu-1)\left(\frac{1}{2} \left(\displaystyle\int_{D}\varphi^{\frac{\mu}{\mu-p}}(x)dx \right)^{\frac{p-\mu}{p}}+\delta \left(\displaystyle\int_{D}\varphi ^{\frac{n}{n-1}}(x)dx\right)^{1-n}\right)}\Bigg). \nonumber
	\end{align}
	 By performing the transformation $s \mapsto \frac{4t}{\rho^{2}_{1}},$ we get
	 \begin{align}
	 &\mathbb{P}\left( \sigma^{\ast}_{1}=\infty \right)\nonumber\\
	 &=\mathbb{P} \Bigg(\frac{4}{\rho^{2}_{1}} \int_{0}^{\infty}e^{2W(t)-(\lambda_{1}+\Lambda)(\mu-1)\frac{4t}{\rho^{2}_{1}}}dt < \frac{J^{1-\mu}(0)}{(\mu-1)\left(\frac{1}{2} \left(\displaystyle\int_{D}\varphi^{\frac{\mu}{\mu-p}}(x)dx \right)^{\frac{p-\mu}{p}}+\delta \left(\displaystyle\int_{D}\varphi ^{\frac{n}{n-1}}(x)dx\right)^{1-n}\right)}\Bigg) \nonumber\\
	 &=\mathbb{P} \Bigg(\frac{4}{\rho^{2}_{1}} \int_{0}^{\infty}e^{2(W(t)-\theta_{1}t)}dt < \frac{J^{1-\mu}(0)}{(\mu-1)\left(\frac{1}{2} \left(\displaystyle\int_{D}\varphi^{\frac{\mu}{\mu-p}}(x)dx \right)^{\frac{p-\mu}{p}}+\delta \left(\displaystyle\int_{D}\varphi ^{\frac{n}{n-1}}(x)dx\right)^{1-n}\right)}\Bigg), \nonumber
	 \end{align}
	 where $\theta_{1}=\displaystyle\frac{2(\lambda_{1}+\Lambda)(\mu-1)}{\rho^{2}_{1}}.$ By Lemma \ref{ll1}, we have 
	 \begin{align}
	 \mathbb{P}\left( \sigma^{\ast}_{1}=\infty \right)&=\mathbb{P} \Bigg( \frac{1}{X(\theta_{1},1)} < \frac{\rho^{2}_{1}J^{1-\mu}(0)}{2(\mu-1)\left(\frac{1}{2} \left(\int_{D}\varphi^{\frac{\mu}{\mu-p}}(x)dx \right)^{\frac{p-\mu}{p}}+\delta\left(\int_{D}\varphi ^{\frac{n}{n-1}}(x)dx\right)^{1-n}\right)}\Bigg) \nonumber\\
	 &=\mathbb{P} \Bigg( 2(\mu-1)\left(\frac{1}{2} \left(\int_{D}\varphi^{\frac{\mu}{\mu-p}}(x)dx \right)^{\frac{p-\mu}{p}}+\delta\left(\int_{D}\varphi ^{\frac{n}{n-1}}(x)dx\right)^{1-n}\right)<X(\theta_{1},1)\rho^{2}_{1}J^{1-\mu}(0)\Bigg). \nonumber
	 \end{align}
	 \textbf{Case 2:} If $m+n>q>1,$ by proceeding in the same way as in the case 2 of Theorem \ref{thm5.1} and by using \eqref{ss1}, we obtain
	 \begin{align}
	 &\sigma^{\ast}_{2} = \inf \Bigg\{ t \geq 0 : \int_{0}^{t}\left( e^{\eta (q-1)W(s)} \wedge e^{\eta (m+n-1)W(s)} \right) e^{-(\lambda_{1}+\Lambda)(q-1)s}ds\geq a_{1}  \Bigg\}. \nonumber 
	 \end{align}
		Note that $m+n-1 > q-1>0$ and so
		\begin{align}
	&	\int_{0}^{t} \left( e^{\eta (q-1)W(s)} \wedge e^{\eta (m+n-1)W(s)} \right) e^{-(\lambda_{1}+\Lambda)(q-1)s}ds \nonumber\\&= \int_{0}^{t} e^{\eta (m+n-1) W(s)-(\lambda+\Lambda)(q-1) s}  \mathbf{1} _{\{W(s)<0\}} ds+\int_{0}^{t} e^{\eta(q-1) W(s)-(\lambda_{1}+\Lambda)(q-1)s}  \mathbf{1} _{\{W(s) \geq 0\}} ds \nonumber\\
		&\geq \int_{0}^{t} e^{\eta(q-1) W(s)-(\lambda_{1}+\Lambda)(q-1)s}  \mathbf{1} _{\{W(s) \geq 0\}} ds, \nonumber
		\end{align}
		and $$\left\lbrace  \eta(q-1) W(t)-(\lambda_{1}+\Lambda)(q-1)t \geq 0 \right\rbrace \subseteq  \left\lbrace W(t) \geq 0 \right\rbrace,$$ for all $t \geq 0$. We conclude that 
		\begin{align}
	&	\int_{0}^{t} \left( e^{\eta (q-1)W(s)} \wedge e^{\eta (m+n-1)W(s)} \right) e^{-(\lambda_{1}+\Lambda)(q-1)s}ds \nonumber\\
		&\geq \int_{0}^{t} e^{-(\eta(q-1) W(s)-(\lambda_{1}+\Lambda)(q-1)s)}  \mathbf{1} _{\{\eta(q-1) W(s)-(\lambda_{1}+\Lambda)(q-1)s \geq 0\}} ds. \nonumber 
		\end{align}
		Therefore, we define
     \begin{align}
     &\sigma^{\ast \ast}_{2} = \inf \Biggl\{ t \geq 0 : \int_{0}^{t}e^{-(\eta(q-1) W(s)-(\lambda_{1}+\Lambda)(q-1)s)}  \mathbf{1} _{\{\eta(q-1) W(s)-(\lambda_{1}+\Lambda)(q-1)s \geq 0\}}ds\geq  a_{1} \Biggr\}. \nonumber 
     \end{align}	
	 It can be easily seen that $\tau \leq \sigma^{\ast}_{2} \leq \sigma^{\ast \ast}_{2}.$ From the definition of $\sigma^{\ast \ast}_{2}$  and  \cite[Theorem 2.6]{Eug2017}, we have
	 \begin{align}
	 \mathbb{P}\{ \tau<\infty\} &\geq \mathbb{P}\{ \sigma^{\ast\ast}_{2}<\infty\} 
\nonumber\\	 &= \mathbb{P} \Bigg( \int_{0}^{\infty}e^{-(\eta(q-1) W(s)-(\lambda_{1}+\Lambda)(q-1)s)}  \mathbf{1} _{\{\eta(q-1) W(s)-(\lambda_{1}+\Lambda)(q-1)s \geq 0\}}ds \geq a_{1} \Bigg) \nonumber\\
	 &= \int^{\infty}_{a_{1}} (\lambda_{1}+\Lambda)(q-1) \sum_{n \geq 1} \exp \Bigg\{ -\left( \frac{(\eta(q-1))^{2}}{8}j^{2}_{\frac{2 (\lambda_{1}+\Lambda)(q-1)}{(\eta(q-1))^{2}}-1, n}\right)y \Bigg\}dy \nonumber\\
	 &= \frac{8 (\lambda_{1}+\Lambda)(q-1)}{(\eta(q-1))^{2}} \sum_{n \geq 1} \frac{\exp \Bigg\{- \frac{(\eta(q-1))^{2} a_{1}}{8}j^{2}_{\frac{2 (\lambda_{1}+\Lambda)(q-1)}{(\eta(q-1))^{2}}-1, n} \Bigg\}}{j^{2}_{\frac{2 (\lambda_{1}+\Lambda)(q-1)}{(\eta(q-1))^{2}}-1, n}},\nonumber 
	 \end{align}
  where we have used the Monotone Convergence Theorem to obtain the final equality, which completes the proof of the theorem.
	 \end{proof}
	 The following theorem provides an upper bound for the probability of blow-up solution $v(\cdot, \cdot)$ of the equation \eqref{ss1}. 
	 

     \begin{theorem}
    Assume that $p,q,n>1,$ $m \geq 0$ with $m+n \geq q$. If $f$ is a non-negative bounded function and $|D|>k$, then an
    upper bound for the probability of blow-up solution of the equation \eqref{ss1} is given by
    	\begin{align}
    	\mathbb{P}\{ \tau<\infty\}\leq \int_{\widetilde{N}_1}^{\infty} h_{3}(y)dy, \nonumber 
    	\end{align}
    	where 
    	\begin{align}
    	\widetilde{N}_{1}&=\frac{1}{2M(m+n-1) \|f\|_{\infty}^{m+n-1}} -\frac{1}{(\lambda+\Lambda)(q-1)},\ \mbox{here}\ M = \max\left\lbrace |D|,\delta |D| \right\rbrace\ \mbox{and}\ \nonumber\\  h_{3}(y)&=\frac{(2/(\eta(q-1))^{2}y)^{(2\Lambda(m+n-1) /(\eta(q-1))^{2})}}{y \Gamma(2\Lambda(m+n-1)/(\eta(q-1))^{2})} \exp\left(-\frac{2}{(\eta(q-1))^{2}y}\right).\ \nonumber
    	\end{align}
     \end{theorem}	
     \begin{proof}
     By using the same procedure as in Theorem \ref{thm6.1}, we have $\sigma_{\ast} \leq \tau,$ where $\sigma_{\ast}$ is given by 
     \begin{align}
     \sigma_{\ast} = \inf \Bigg\{ t\geq 0 :\int_{0}^{t}\left( e^{\eta \left( q-1\right)  W(r)} \vee e^{\eta (m+n-1) W(r)}\right)  e^{-\Lambda (m+n-1) r}dr \geq  \frac{1}{2M(m+n-1) \|f\|_{\infty}^{m+n-1}} \Bigg\}, \nonumber
     \end{align}
     Note that $1<q\leq m+n$ and so
     \begin{align}
 &   \int_{0}^{t} \left( e^{\eta (q-1)W(s)} \wedge e^{\eta (m+n-1)W(s)} \right) e^{-\Lambda(m+n-1)s}ds\nonumber\\ &= \int_{0}^{t} e^{\eta (m+n-1) W(s)-\Lambda(m+n-1) s}  \mathbf{1} _{\{W(s)<0\}} ds+\int_{0}^{t} e^{\eta(q-1) W(s)-\Lambda(m+n-1)s}  \mathbf{1} _{\{W(s) \geq 0\}} ds \nonumber\\
     &\leq \int_{0}^{\infty} e^{-\Lambda(m+n-1) s} ds +\int_{0}^{t} e^{\eta(q-1) W(s)-\Lambda(m+n-1)s} ds \nonumber\\
     &=\frac{1}{\Lambda(m+n-1)}+\int_{0}^{t} e^{\eta(q-1) W(s)-\Lambda(m+n-1)s} ds,
     \end{align}
     for all $t \geq 0$. Therefore, we define 
     \begin{align}
     \sigma_{\ast \ast}=\inf \Bigg\{ t\geq 0 : \int_{0}^{t} e^{\eta(q-1) W(s)-\Lambda(m+n-1)s} ds &\geq \frac{1}{2M(m+n-1) \|f\|_{\infty}^{m+n-1}} -\frac{1}{\Lambda(m+n-1)}\Bigg\}, \nonumber
     \end{align}
      and $\sigma_{\ast \ast} \leq \sigma_{\ast}\leq \tau$. Using the definition of $\sigma_{\ast \ast}$, we have 
      \begin{align}
      \mathbb{P}\{ \tau<\infty\} \leq \mathbb{P}( \sigma_{\ast \ast}<\infty) &= \mathbb{P}\Bigg(\int_{0}^{\infty} e^{\eta(q-1) W(s)-\Lambda(m+n-1)s}  ds \geq \widetilde{N}_{1}\Bigg)\nonumber\\
      &= \mathbb{P}\Bigg( \int_{0}^{\infty} \exp\{2 W_{s}^{(\hat{\alpha})}\}ds \geq \widetilde{N}_{1} \Bigg), \nonumber      
      \end{align}
      where $W_{s}^{(\hat{\alpha})}:=W(s)-\hat{\alpha}s$ and $\hat{\alpha} =\frac{\Lambda(m+n-1)}{\eta(q-1)}.$ By performing the transformation $s\mapsto \frac{4t}{(\eta(q-1))^{2}}$ and setting $\nu_{1} = \frac{2\hat{\alpha}}{\eta(q-1)}$,  we get
      \begin{align}
      \mathbb{P}\{ \tau<\infty\}
      \leq \mathbb{P}\Bigg(\frac{4}{(\eta(q-1))^{2}} \int_{0}^{\infty} \exp\{2 W_{t}^{(\nu_{1})}\}dt \geq \widetilde{N}_1 \Bigg).\nonumber
      \end{align}
      	Therefore, we deduce 
      	\begin{align}
      	\mathbb{P}\{ \tau<\infty\}\leq \int_{\widetilde{N}_1}^{\infty} h_{3}(y)dy, \nonumber 
      	\end{align}
      which completes the proof.      
     \end{proof}

	 \section{Appendix}\label{app}
	 In this appendix,  we provide  a  comparison result  which is used in Section \ref{sec4} (see  \cite[Theorem 1.3]{teschl} also).  From \eqref{511}, we have 
	  \begin{align}
	  J'(t)& \geq  \left(-\lambda_{1}+\gamma\right)J(t)+\mathit{\widetilde{N}}e^{\eta\left(\mu-1\right)  B^{H}(t)}J^{\mu}(t), \nonumber
	  \end{align}
	  and
	  \begin{align}
	  \frac{d I(t)}{dt} = \left(-\lambda_{1}+\gamma\right)I(t)+\mathit{\widetilde{N}}e^{\eta\left(\mu-1\right)  B^{H}(t)}I^{\mu}(t),\ I(0)=J(0). \nonumber
	  \end{align}
	 We have to show that $J(t) \geq I(t),$ for $ t \in [0, \tau).$ Suppose that $J(t) <I(t),$ for some $t \in [t_{0}, t_{0}+\varepsilon),\ \varepsilon>0.$  Let $\triangle(t)=I(t)-J(t).$ Then an application of Taylor's formula yields 
	 \begin{align}
	 \triangle'(t)={I}'(t)-{J}'(t) &\leq (-\lambda_{1}+\gamma) \triangle(t)+\mathit{\widetilde{N}}e^{\left(\mu-1\right) \eta B^{H}(t)}\left(I^{\mu}(t)-J^{\mu}(t) \right) \nonumber \\
	 & \leq (-\lambda_{1}+\gamma) \triangle(t)+\mu \kappa \mathit{\widetilde{N}} \triangle(t) \leq \kappa_{1} \triangle(t), \nonumber
	 \end{align} 
	 where $\kappa=\sup \limits_{t \in [0, \tau]}e^{\left(\mu-1\right) \eta B^{H}(t)}\sup \limits_{t \in [0, \tau]}(|J(t)|+|I(t)|)^{\mu-1}$ and $\kappa_{1}= \max \{ (-\lambda+\gamma), \mu \kappa \mathit{\widetilde{N}} \}.$ Therefore, we obtain $\dot{\triangle}(t)-\kappa_{1} \triangle(t) \leq 0$ and this implies that $e^{-\kappa_{1}t}\triangle (t) \leq 0,$ so that $\triangle(t) \leq 0$ which is a contradiction, since $\triangle(t)>0,$ for all $t \in [t_{0}, t_{0}+\varepsilon)$. Hence $J(t) \geq I(t),$ for all $t \in [0, \tau).$
	
  \vskip 0.2 cm
    \medskip\noindent
	{\bf Acknowledgements:} The first author is supported by the University Research Fellowship of Periyar University, India. M. T. Mohan would  like to thank the Department of Science and Technology (DST), India for Innovation in Science Pursuit for Inspired Research (INSPIRE) Faculty Award (IFA17-MA110). The third author is supported by the Fund for Improvement of Science and Technology Infrastructure (FIST) of DST (SR/FST/MSI-115/2016).

\end{document}